\documentclass[12pt, aps, showpacs, pra, superscriptadaddress, eqsecnum,notitlepage,USLetter,nofootinbib]{revtex4-1}

\usepackage{titlesec}
\usepackage{hyperref}
\usepackage{amssymb}
\usepackage{amsthm}
\usepackage{ulem}
\usepackage{amsmath}
\usepackage{amsfonts}
\usepackage{epstopdf}
\usepackage{mathtools,nccmath,relsize}
\usepackage[dvipsnames]{xcolor}
\usepackage{bbm}
\usepackage{enumerate}
\usepackage{calrsfs}
\usepackage{bigints}
\usepackage{autonum}
\newtheorem{Th}{Theorem}[section]
\newtheorem{Prop}[Th]{Proposition}
\newtheorem{Lem}[Th]{Lemma}

\theoremstyle{Remark}
\newtheorem{Rmk}[Th]{Remark}
\theoremstyle{Definition}

\theoremstyle{Corollary}

\theoremstyle{plain} 
\newcommand{\thistheoremname}{}
\newtheorem{genericthm}[Th]{\thistheoremname}
\newenvironment{namedthm}[1]
  {\renewcommand{\thistheoremname}{#1}%
   \begin{genericthm}}
  {\end{genericthm}}

\allowdisplaybreaks

\newcommand{\weak}{\rightharpoonup}
\newcommand{\ZT}{\mathcal{Z}}
\newcommand{\EV}{\mathcal{E}_V}
\newcommand{\Ez}{\mathcal{E}_0}
\newcommand{\EVn}{\mathcal{E}_{V_n}}
\newcommand{\ETF}{\mathcal{E}_{V_{TF}}}
\newcommand{\RRR}{\mathbb{R}^3}
\newcommand{\R}{\mathbb{R}}
\newcommand{\Hone}{\mathcal{H}^1(\RRR)}
\newcommand{\eps}{\epsilon}
\newcommand{\Om}{\Omega}
\newcommand{\nnn}{\nonumber}
\newcommand{\NN}{\mathbb{N}}
\newcommand{\Rnb}{\overline{R}_n}
\newcommand{\Rnz}{R_n^0}
\newcommand{\rnb}{\overline{\rho}_n}

\begin{document}

\baselineskip=15pt

\title{Mass splitting in the Thomas-Fermi-Dirac-von Weizs\"acker model with background potential
}

\author{Lorena Aguirre Salazar}
\affiliation{Department of Mathematics and Statistics, McMaster University, 1280 Main Street West, Hamilton, Ontario, Canada}
\author{Stan Alama}
\affiliation{Department of Mathematics and Statistics, McMaster University, 1280 Main Street West, Hamilton, Ontario, Canada}
\author{Lia Bronsard}
\affiliation{Department of Mathematics and Statistics, McMaster University, 1280 Main Street West, Hamilton, Ontario, Canada}

\begin{abstract}

\begin{center}\textbf{Abstract}\end{center}

We consider minimization problems of the Thomas-Fermi-Dirac-von Weizs\"{a}cker (TFDW) type, in which the Newtonian potential is perturbed by a background potential satisfying mild conditions and which ensures the existence of minimizers. We describe the structure of minimizing sequences for those variants, and obtain a more precise characterization of patterns in minimizing sequences for the TFDW functionals regularized by long-range perturbations.
\end{abstract}

\maketitle

\section{Introduction}

In this paper we are concerned with energy functionals which include the Thomas-Fermi-Dirac-von Weiz\"{a}cker (TFDW) model, a physical model describing ground state electron configurations of many-body systems.  
More precisely, we consider the following variational problem

\begin{align}\begin{split}\label{TheModel}I_V(M):=\inf\{\mathcal{E}_V(u):u\in\mathcal{H}^1(\RRR),\vert\vert u\vert\vert_{\mathcal{L}^2(\RRR)}^2=M\},\end{split}\end{align}
where the energy $\mathcal{E}_V$ is defined as
\begin{align}\mathcal{E}_V(u)&:=\int_{\RRR}\left[\vert\nabla u( x )\vert^2+c_1\vert u( x )\vert^{\frac{10}{3}}-c_2\vert u( x )\vert^{\frac{8}{3}}-V( x ) u^2(x) \right]d x \\
&\qquad\qquad\quad
+\frac{1}{2}\int_{\RRR}\int_{\RRR}\frac{ u^2( x ) u^2( y ) }{\vert x - y \vert}d x d y ,\end{align}
with $c_1,c_2>0$, 
\begin{align}\begin{split}\label{condiV}V\geq0,\quad V\in\mathcal{L}^{\frac{3}{2}}(\RRR)+\mathcal{L}^{\infty}(\RRR),\textit{ and }\lim_{\vert x \vert\to\infty}V( x )=0.\end{split}\end{align}
The conditions above ensure $I_V$ is finite, $\mathcal{E}_V$ is coercive in $\mathcal{H}^1(\R^3)$ on the constraint set, and 
\begin{align}\begin{split}u\in\mathcal{H}^1(\RRR)\mapsto \int_{\RRR}V( x ) u^2(x) d x \end{split}\end{align}
is weakly continuous.

The TFDW model corresponds to the choice
\begin{align}\begin{split}\label{PotentialTFDW}V_{TF}( x )=\sum_{k=1}^K\frac{\alpha_k}{\vert x - r _k\vert}, \end{split}\end{align}
with $K\in\mathbb{N}$, $\{\alpha_k\}_{k=1}^K\subset\mathbb{R}^+$ and $\{ r _k\}_{k=1}^K\subset\RRR$ all fixed. In this case, $\mathcal{E}_V(M)$ is to be thought of as the energy of a system of $M$ electrons interacting with $K$ nuclei. Each nucleus has charge $\alpha_k>0$ and it is fixed at a position $ r _k$. 
The total nuclear charge is denoted
$$   \ZT=\sum_{k=1}^K \alpha_k >0,  $$
and plays a key role in existence results
 (see the works of Frank, Nam, and Van Den Bosch~\cite{FrankNamBosch}, and Lieb~\cite{Lieb} for a survey). 

  In this paper we explore the structure of minimizing sequences for $I_V$, with $V$ chosen to be a perturbation of the molecular potential $V_{TF}$.  
Despite the coercivity of the problem, existence of a minimizer for $I_V$ is a highly nontrivial problem, due to a lack of compactness at infinity.
For the unperturbed TFDW problem, $V=V_{TF}$, Lions~\cite{Lions} proved that there exists a minimizer if $M\leq \ZT$, and Le Bris~\cite{LeBris} extended this result to $M\leq  \ZT+\epsilon$ for some $\epsilon=\epsilon(\ZT)>0$. In regard to non-existence, Nam and Van Den Bosch proved there are no minimizers if both $M$ is sufficiently large and $\ZT$ is sufficiently small, and Frank, Nam, and Van Den Bosch~\cite{NamVBosch} proved nonexistence of a minimizer for $M>\ZT+C$, for some universal $C>0$, in the case there is only one nucleus (i.e., $K=1$ in \eqref{PotentialTFDW}.)

There is a special class of potentials $V$ for which the existence problem for $I_V$ is completely understood.  We say $V$ is a \underbar{long-range} potential if it satisfies \eqref{condiV} and 
\begin{equation}\label{Long}
  \liminf_{t\to\infty}t\left(\inf_{\vert x \vert=t}V( x )\right)=\infty.
\end{equation}
For example, the homogeneous potentials $V^\nu( x )= \vert x\vert^{-\nu}$ are of long-range for $0<\nu<1$.
For long-range potentials \eqref{Long}, Alama, Bronsard, Choksi, and Topaloglu \cite{AlamaBronsardChoksiTopalogluLongRange} showed that $I_V(M)$ is attained for {\em every} $M>0$.  Thus, we may perturb the TFDW potential via a long-range potential of the form $V^\nu$, and think of this as a regularized'' version of TFDW.  We thus define a family of long-range potentials,
\begin{equation}\label{VZ}  V_Z( x )= V_{TF}( x ) + {Z\over \vert x\vert^\nu}, \qquad 0<\nu<1,  \end{equation}
with parameter $Z>0$. By taking a sequence $Z_n\to 0$ we recover the TFDW model, but via a special minimizing sequence $u_n$ composed of minimizers of the long-range problem, $\EVn(u_n)=I_{V_n}$.
A special role is played by the minimization problem $I_0$, that is with potential $V\equiv 0$, which is the ``energy at infinity'' obtained by translating $u(\cdot+x_n)$ with $|x_n|\to\infty$.  The existence properties for $I_0(M)$ are analogous to those of $I_{V_{TF}}$:  the minimizer exists for sufficiently small $M>0$ (see \cite[Lemma 9 (iii)]{NamVBosch},) and there is no minimizer for all large $M$ (see \cite{LuOttoNon-existence}.)  

It will be convenient to introduce the following set of values of the constrained mass $M$ in $I_V(M)$:
$$    \mathcal{M}_V:=\left\{ M>0 \ \bigl| \ \text{$I_V(M)$ has a minimizer $u\in \Hone$}, \ \int_{\RRR} u^2 = M\right\}.
$$
It is an open question to determine whether $\mathcal{M}_V$ is an interval, for any choice of potential $V$.  

In case $u\in\Hone$ attains the minimum in $I_V$ (respectively, $u_0\in\Hone$ attains the minimum in $I_0$), the minimizers will satisfy the PDEs,
\begin{gather}
\begin{split}\label{PDEsMin1}-\Delta u+\frac{5}{3}c_1u\vert u\vert^{\frac{4}{3}}-\frac{4}{3}c_2u\vert u\vert^{\frac{2}{3}}-Vu+\left(\vert u\vert^2\star\vert\cdot\vert^{-1}\right)u=\mu u\end{split}
\\ 
\begin{split}\label{PDEsMin2}\ \ -\Delta u_0+\frac{5}{3}c_1u_0\vert u_0\vert^{\frac{4}{3}}-\frac{4}{3}c_2u_0\vert u_0\vert^{\frac{2}{3}}+\left(\vert u_0\vert^2\star\vert\cdot\vert^{-1}\right)u_0=\mu u_0,\end{split}\end{gather}
with Lagrange multiplier $\mu$ induced by the mass constraint.  

As mentioned above, the existence question is complicated by noncompactness due to translations of mass to infinity.  However, minimizing sequences may be characterized using a general Concentration-Compactness structure (see \cite{LionsConcentrationPart1}, \cite{Lions}).


\begin{namedthm}{Concentration Theorem} 
\label{TheoremSplit} Let $\{u_n\}_{n\in\mathbb{N}}$ be a minimizing sequence for $I_V(M)$ where $V$ satisfies \eqref{condiV}. Then, there exist a number $N\in\mathbb{N}\cup\{0\}$, masses $\{m^{i}\}_{i=0}^N\subset\mathbb{R}^+$, translations $\{ x _n^0\}_{n\in\mathbb{N}}$,$\dots$,$\{ x _n^{N}\}_{n\in\mathbb{N}}\subset\RRR$, and functions $\left\{u^{i}\right\}_{i=0}^N\subset \mathcal{H}^1(\RRR)$ such that, up to a subsequence,
\begin{align}\begin{split}\label{ca1}u_n(\cdot)-\sum_{i=0}^{N}u^{i}\left(\cdot- x _n^{i}\right)\to0\textit{ in $\mathcal{H}^1(\RRR)$},\end{split}\end{align}
\begin{equation}\begin{gathered}\label{splitting1} I_V(m^{0})=\mathcal{E}_V(u^{0}),\quad I_0(m^{i})=\mathcal{E}_0(u^{i}),i>0,
\\  
\text{where} \quad \vert\vert u^{i}\vert\vert_{\mathcal{L}^2(\RRR)}^2=m^{i};
\end{gathered}\end{equation}
\begin{align}\begin{split}\label{splitting2}\sum_{i=0}^Nm^{i}=M,\quad I_V(m^0)+\sum_{i=0}^NI_0(m^{i})=I_V(M),\end{split}\end{align}
\begin{gather}\label{xblow}\vert  x _n^{i}- x _n^{j}\vert\to\infty,\quad i\neq j.
\end{gather}
The functions $u^i$ satisfy \eqref{PDEsMin1} for $i=1,\dots,N$, and $u^0$ satisfies \eqref{PDEsMin2}, each with the \underbar{same} Lagrange multiplier $\mu\le 0$.

Moreover, if $V\not\equiv0$, then we can take $ x _n^0={0}$.
\end{namedthm}


If a minimizer exists then no splitting is necessary, and there exist minimizing sequences with $N=0$.  This occurs for $V_{TF}$ when the mass is not much larger than the total charge, $M\leq \ZT + \epsilon$ (by \cite{LeBris}), for instance, or for any $M>0$ in the class of long-range potentials \eqref{Long}.  However, for TFDW with large mass $M$ we expect splitting, but the pieces resulting from noncompactness must each minimize $I_V$ or $I_0$ for its given mass, that is, 
$$ m^0\in \mathcal{M}_V, \quad m^i\in \mathcal{M}_0,i>0.  $$

The basic idea behind the result is very elegant and intuitive.  Minimizing sequences $u_n$ for $I_V(M)$ may lose compactness due to splitting into widely spaced components, each of which tends to a minimizer of $I_V$ or (for those components which translate off to infinity) $I_0$.  Asymptotically, all of the mass $M$ is accounted for by this splitting.  Although the pieces eventually move infinitely far away, they retain some information of the original minimization problem in that they share the same Lagrange multiplier.

Concentration results of this type have appeared in many papers.  For TFDW, a very similar result is outlined (although with possibly infinitely many components $u^i$,) in \cite{Lions} and a proof of the exact decomposition of energy \eqref{splitting2} for the case $V\equiv 0$ is given in \cite[Lemma 9]{NamVBosch}.  Since this Concentration Theorem is central to the statements and proof of our results we provide a proof in Appendix A.  The finiteness of the components is a result of the concavity of the energy $\EV$ for small masses, which we prove in Appendix B.

\bigskip

For perturbations of TFDW we obtain more precise information on the splitting structure.  In particular, when mass splits off to infinity, the piece which remains localized must have mass $m^0\ge \ZT$, the total nuclear charge.

\begin{Th}\label{TheoremSplitNewtonian}  Assume $V$ satisfies \eqref{condiV} and 
\begin{equation}\label{PotGeqTFDW}V( x )\geq  V_{TF}( x )=\sum_{k=1}^K\frac{\alpha_k}{\vert x - r _k\vert}, \quad \textit{ a.e. in $\RRR$},\end{equation}
for some $K\in\mathbb{N}$, $\{\alpha_k\}_{k=1}^K\subset\mathbb{R}^+$ and $\{ r _k\}_{k=1}^K\subset\RRR$.

Then, with the notation of Theorem~\ref{TheoremSplit}, for any minimizing sequence $\{u_n\}_{n\in\mathbb{N}}$ of $I_{V}(M)$, either $M\in\mathcal{M}_V$ or splitting occurs with $m^0\ge \ZT=\sum_{k=1}^K\alpha_k$.
\end{Th}
Heuristically, this is a satisfying result:  after splitting, the nuclei should still capture as many electrons as the total nuclear charge $\ZT$.  One might expect that it should be able to retain strictly more, to form a negatively charged ion.

\medskip

 Finally, we consider in greater detail the loss of compactness which may occur for the long-range regularized families $V_n$ satisfying \eqref{formulaVn} with $0<\nu<1$.  First, minimizers of $I_{V_n}(M)$ form a minimizing sequence for $I_{TF}(M)$, so when $M$ is large compared to $\ZT$, compactness is lost and mass splits off to infinity as described in Theorem~\ref{TheoremSplit}.

\begin{Prop}\label{TSN2} Let $Z_n\to 0$ and
\begin{align}\begin{split}\label{formulaVn}
V_n({x})= V_{TF}({x})+\frac{Z_n}{\vert{x}\vert^\nu}
,\end{split}\end{align}
with $0<\nu<1$, and $\ZT=\sum_{k=0}^K \alpha_k$.  Let $u_n$ minimize $I_{V_n}(M)$, $n\in\mathbb{N}$.  
Then, 
\begin{enumerate}
\item[(i)]  $\{u_n\}_{n\in\mathbb{N}}$ is a minimizing sequence for $I_{V_{TF}}$.
\item[(ii)]  
Either $M\in\mathcal{M}_{V_{TF}}$ or splitting occurs with $m^0\geq \ZT$. 
\end{enumerate}
 \end{Prop}
 
  The nonlocal term in $\EVn$ exerts a repulsive effect on the components $u^i$, while the vanishing long-range potential provides some degree of containment.  The combination of attractive and repulsive terms generally leads to pattern formation, at a scale determined by the relative strengths of the competitors.  This phenomenon has been identified in nonlocal isoperimetric problems (such as the Gamow liquid drop model; see \cite{CP, AlamaBronsardChoksiTopalogluDroplet}.)

However, for potentials  $V_n$ of the form \eqref{formulaVn}, the interactions between the fleeing components $u^i$ appear in the energy at order $Z_n^{\frac{1}{1-\nu}}$.  Thus,
 we require some information about the spatial decay of the minimizers of $\EVn$ away from the centers of the support in order to control the errors in an expansion of the energy in terms of $Z_n\to 0$.  In the liquid drop problems, the splitting is into characteristic functions of disjoint bounded domains, and this issue does not arise.  In order to calculate interactions we require exponential decay of the solutions, which is connected to the delicate question of the Lagrange multiplier $\mu$.  In particular, we obtain exponential decay when $\mu<0$, 
$$   |u(x)|\le C e^{-\lambda|x|},  $$
for any $0<\lambda<\sqrt{-\mu}$.
As the energy value $I_V(M)$ is strictly decreasing in $M$, we have $\mu\le 0$ and in fact 
we would expect that $\mu<0$ should hold, if not always, at least for all but a residual set of $M$.  It is an open question whether $\mu<0$ holds whenever $M\in\mathcal{M}$.  The strict negativity is known for the cases $V\equiv 0$ with sufficiently small mass, or with $V=V_{TF}$ with $M<\ZT + \kappa$ with $\kappa=\kappa(V_{TF})>0$; see Proposition~\ref{muneg}.

 We may now state our result on the distribution of masses in the case of splitting.  First, we define  
%
\begin{equation}\label{Mstar}
\mathcal{M}^*_V=\left\{ M\in\mathcal{M}_V \ : \ \text{every minimizer $u$ of $I_V(M)$ satisfies \eqref{PDEsMin1} with $\mu<0$.}
\right\}
\end{equation}

\begin{Th}\label{TheoremLocation} Let $u_n$ be minimizers of $I_{V_n}(M)$ with $V_n$ satisfying \eqref{formulaVn} with $0<\nu<1$ and $Z_n\to 0$. Let
$N$, $\{m^i\}_{i=0}^N$ and $\{ x ^0_n\}_{n\in\mathbb{N}}$,$\dots$,$\{ x ^N_n\}_{n\in\mathbb{N}}$ be as in  Theorem~\ref{TheoremSplit}.  Assume $N\ge 1$ and 
 $m^0\in\mathcal{M}_{V_{TF}}^*$. Then, up to a subsequence and relabeling, either
 \begin{enumerate}
 \item[(i)] $m^0>Z$ and
\begin{align}\begin{split}\label{convergencetoy1}Z_n^{\frac{1}{1-\nu}} x _n^{i}\to y ^{i},\quad {i=1,\dots,N,}\end{split}\end{align}
where $(y ^1,\dots, y ^{N})$ minimizes the interaction energy
\begin{equation}F_{N,(m^{0},m^1,\dots,m^{N})}(w ^1,\dots, w ^{N}):=\sum_{1\le i<j}\frac{m^{i}m^{j}}{\vert  w ^{i}- w ^{j}\vert}
+\left(m^0-\ZT\right)\sum_{i=1}^N\frac{m^i}{\vert w ^i\vert}-\sum_{i=1}^N\frac{m^{i}}{\vert  w ^{i}\vert^{\nu}}\end{equation}
over 
\begin{align}\begin{split}\Sigma_N:=\left\{(w ^1,\dots, w ^{N})\in(\mathbb{R}^3\setminus\{0\})^{N}: w ^{i}\neq w^j\right\};\end{split}\end{align}
or,
\item[(ii)] $m^0=Z$, $Z_n^{\frac{1}{1-\nu}}x_n^1\to0$ and if $N\ge 2$ we have:
\begin{align}\begin{split}\label{convergencetoy2}Z_n^{\frac{1}{1-\nu}} x _n^{i}\to y ^{i},\quad {i=2,\dots,N,}\end{split}\end{align}
where $(y ^2,\dots, y ^{N})$ minimizes the interaction energy
\begin{equation}\overline{F}_{N,(m^{1},m^2,\dots,m^{N})}(w ^2,\dots, w ^{N}):=\sum_{{2\leq i<j}}\frac{m^{i}m^{j}}{\vert  w ^{i}- w ^{j}\vert}
+m^1\sum_{i=2}^N\frac{m^i}{\vert w ^i\vert}-\sum_{i=2}^N\frac{m^{i}}{\vert  w ^{i}\vert^{\nu}}\end{equation}
over 
\begin{align}\begin{split}\overline{\Sigma}_N:=\left\{(w ^2,\dots, w ^{N})\in(\mathbb{R}^3\setminus\{0\})^{N-1}: w ^{i}\neq w^j\right\}.\end{split}\end{align}
\end{enumerate}
\end{Th}

\begin{Rmk}
\begin{enumerate}
\item 
The degenerate case $m^0=\mathcal{Z}$ is very delicate, as the term measuring the repulsion between the weakly convergent component  supported near zero and the diverging pieces is nearly exactly cancelled by the attractive effect of the nuclear potential $V_{TF}$.  Thus, the error terms in the expansion of the energy may exceed the principal term creating a net repulsion (or attraction) to the nuclei which is difficult to estimate.  For instance, if $N=1$ and only one component splits to infinity then all we can say when $m^0=\mathcal{Z}$ is that it diverges at a rate much slower than $Z^{-{1\over 1-\nu}}$.  In some sense, there is no natural scale for its interaction distance to the nuclei.  For this reason, we believe that in fact $m^0>\mathcal{Z}$, but have no proof of this conjecture.
\item If $m^0=\ZT$, then $m^0\in\mathcal{M}_{V_{TF}}^*$ automatically (see Proposition \ref{muneg} (ii).)
\item The proof of the compactness of all minimizing sequences of $\inf F_{N,(m^{0},m^1,\dots,m^{N})}$ and $\inf \overline{F}_{N,(m^{1},m^2,\dots,m^{N})}$ follows with little modification from the proof of  \cite[Proposition 8]{AlamaBronsardChoksiTopalogluDroplet}.
\item  By Theorem~\ref{TheoremSplit}, each of the components $u^i$ shares the same Lagrange multiplier $\mu$, and hence it suffices that any one of the components satisfy \eqref{PDEsMin1} with $\mu<0$.
\item We do not know whether the condition $\mu<0$ could be improved. We use $\mu<0$ for uniform exponential decay of the functions $u_n$ away from $ x _n^i$, but some weaker uniform decay away from the mass centers may be sufficient. However, it is unclear how rapidly minimizers of \eqref{TheModel} decay when $\mu=0$.
\end{enumerate}\end{Rmk}

Finally, we note that the specific choice of powers $p=\frac{10}{3}$ and $q=\frac83$ in the nonlinear potential well $W(u)=c_1 |u|^p - c_2 |u|^q$ are physically appropriate for the TFDW model, but from the point of view of analysis other choices are possible.  Indeed, most of the results of this paper may be extended to the case $2<q<3$ and $q<p<6$.  However, for $q>3$ the behavior of minimizers may be substantially different:  in such case $I_V(M)=I_V^{p,q}(M)$ may vanish identically, and minimizers may never exist for any $M>0$.  (See Lions~\cite{Lions} for various examples.)  Thus, it is not sufficient to have potentials $W$ with a ``double well'' structure to observe the properties of TFDW minimizers; the relationship between the powers appearing in the functional is of importance, as well.

\section{Boundedness and decay of minimizers}

In this section we prove various basic properties of $I_V(M)$ and its minimizers, and we discuss the role of the Lagrange multiplier in the decay of solutions.

The following properties are well-known for variational problems of the form \eqref{TheModel}:

\begin{Prop}\label{properties1}  Let $V$ satisfy \eqref{condiV}.
\begin{enumerate}
\item[(i)] 
For any $M>0$, $I_V(M)<0$, and is strictly decreasing in $M$.  
\item[(ii)] The following ``binding inequality'' holds for any $0<m<M$:
\begin{equation}\label{binding}   I_V(M)\le I_V(m) + I_0(M-m).  
\end{equation}
\item[(iii)]
If $I_V(M)$ is attained at $u\in\Hone$, then $u$ solves \eqref{PDEsMin1} with Lagrange multiplier $\mu\le 0$ and we may take $u\geq0$ in $\RRR$. It is possible to choose $u>0$ if $V=V_{TF}$ or $V=V_Z$ as defined in \eqref{VZ}.
\end{enumerate}
\end{Prop}

\begin{proof} Statements (i) and (ii) can be proven as Lemma 5 was in Nam and Van Den Bosch \cite{NamVBosch}. In regard to (iii), \eqref{PDEsMin1} corresponds to the Euler-Lagrange equation associated with $I_V(M)$, while $\mu\le 0$ due to $I_V(M)$ being decreasing in $M$. We may take $u\geq0$ in $\RRR$ as $I_V(M)=\mathcal{E}_V(u)=\mathcal{E}(\vert u\vert)$.  Finally, for potentials of the form $V_{TF}$ or as perturbed in \eqref{VZ} the positivity of minimizers follows from the Harnack inequality.\end{proof}

When the Lagrange multiplier $\mu<0$ we obtain exponential decay (see (66) in Lions \cite{Lions}):  for all $0<\lambda<\sqrt{-\mu}$, there exists a constant $C$ with
\begin{equation} \label{expdecay}
|u(x)| + |\nabla u(x)|\le  Ce^{-\lambda x}, \quad\textit{a.e. in }\RRR
\end{equation}
A categorization of the potentials $V$ and masses $M$ for which $\mu<0$ remains an important open question.  The following proposition gives various criteria under which the Lagrange multiplier $\mu<0$.
We recall the definition of $\mathcal{M}_{V}^*$ in \eqref{Mstar}.

\begin{Prop}\label{muneg} 
\begin{enumerate}
\item[(i)]  For $V\equiv 0$, $\exists M_0>0$ so that if $M<M_0$ then $M\in\mathcal{M}_{V}^*$;
\item[(ii)] For $V\geq V_{TF}$ satisfying \eqref{condiV}, $\exists \kappa=\kappa(\ZT)>0$ so that if $M<\ZT+\kappa$ then $M\in\mathcal{M}_{V_{TF}}^*$;
\item[(iii)]    For $V$ with long-range decay \eqref{Long}, every $M\in \mathcal{M}^*_V$.  
\item [(iv)]  For $V$ satisfying \eqref{condiV} such that
\begin{align}\label{infimumisnega}E:=\inf\left\{\int_{\RRR}(\vert\nabla u\vert^2-Vu^2)dx : u\in\mathcal{H}^1(\RRR),\vert\vert u\vert\vert_{\mathcal{L}^2(\RRR)}=1\right\}<0,\end{align}
there exists $M_V>0$ so that if $M<M_V$ then $M\in\mathcal{M}_{V}^*$;
\item[(v)] For\begin{align}\begin{split}\label{Valpha}V( x )=\sum_{k=1}^K\frac{\alpha_k}{\vert x - r _k\vert^\tau},\textit{ a.e. in $\RRR$,}\end{split}\end{align}
with $\{\alpha_k\}_{k=1}^K\subset\mathbb{R}^+$, $\{ r _k\}_{k=1}^K\subset\RRR$, and $0<\tau<2$, there exists $M_V>0$ so that if $M<M_V$ then $M\in\mathcal{M}_{V}^*$
\end{enumerate}
\end{Prop}

\begin{proof}
To verify (i), suppose that $u$ is a minimizer. Equation \eqref{PDEsMin2} corresponds to the Euler-Lagrange equation associated with $I_0(M)$. Regarding the strict negativity of $\mu$, note that from \eqref{PDEsMin2}, 
\begin{multline}\label{eqmuM}\mu  M=\int_{\RRR}\vert\nabla  u ( x )\vert^2d x +\frac{5}{3}c_1\int_{\RRR}\vert  u ( x )\vert^{\frac{10}{3}}d x -\frac{4}{3}c_2\int_{\RRR}\vert  u ( x )\vert^{\frac{8}{3}}d x \\
+\int_{\RRR}\int_{\RRR}\frac{ u^2( x ) u^2( y ) }{\vert x - y \vert}d x d y.\end{multline}
Moreover, since $I_0(M)=\mathcal{E}_0(u)$ and $\vert\vert \sigma^{\frac{3}{2}}u(\sigma\cdot)\vert\vert_{\mathcal{L}^2(\RRR)}^2=M$ for all $\sigma>0$,
\begin{align}\begin{split}\label{derivativezero}0&=\left.\frac{d}{d\sigma}[\mathcal{E}_0(\sigma^{\frac{3}{2}} u (\sigma\cdot))]\right\vert_{\sigma=1}\\
&=2\int_{\RRR}\vert\nabla  u ( x )\vert^2d x +2c_1\int_{\RRR}\vert  u ( x )\vert^{\frac{10}{3}}d x -c_2\int_{\RRR}\vert  u ( x )\vert^{\frac{8}{3}}d x+\frac{1}{2}\int_{\RRR}\int_{\RRR}\frac{ u^2( x ) u^2( y ) }{\vert x - y \vert}d x d y ,\end{split}\end{align}
or, equivalently, 
\begin{multline}\label{gres}\frac{5}{3}c_1\int_{\RRR}\vert  u ( x )\vert^{\frac{10}{3}}d x =  -\frac{5}{3}\int_{\RRR}\vert\nabla  u ( x )\vert^2d x +\frac{5}{6}c_2\int_{\RRR}\vert  u ( x )\vert^{\frac{8}{3}}d x \\ 
-\frac{5}{12}\int_{\RRR}\int_{\RRR}\frac{ u^2( x ) u^2( y ) }{\vert x - y \vert}d x d y .
\end{multline}

Then, inserting \eqref{gres} into \eqref{eqmuM} gives
\begin{align}\begin{split}\mu M=-\frac{2}{3}\int_{\RRR}\vert\nabla  u ( x )\vert^2d x -\frac{1}{2}c_2\int_{\RRR}\vert  u ( x )\vert^{\frac{8}{3}}d x +\frac{7}{12}\int_{\RRR}\int_{\RRR}\frac{ u^2( x ) u^2( y ) }{\vert x - y \vert}d x d y.\end{split}\end{align}
We conclude by noting that, by Hardy-Littlewood's inequality and the interpolation inequality in Lebesgue spaces,
\begin{align}\begin{split}\int_{\RRR}\int_{\RRR}\frac{ u^2( x ) u^2( y ) }{\vert x - y \vert}d x d y \leq C M^{\frac{2}{3}}\int_{\RRR}\vert u( x )\vert^{\frac{8}{3}}d x .\end{split}\end{align}
We observe that the Pohozaev identity associated with \eqref{PDEsMin2} does not bring new information about $\mu$.

Statement (ii)  is  Theorem 1 by Le Bris \cite{LeBris}, and (iii) is Theorem~2 in  Alama-Bronsard-Choksi-Topaloglu \cite{AlamaBronsardChoksiTopalogluLongRange}.

Statement (iv) follows by the same reasoning as in the proof of Lions \cite[Corollary II.2]{Lions}.   
Finally, (v)  is a consequence of part (iv). Indeed, note that the $\mathcal{L}^2$-norm in $\RRR$ is invariant under the transformation $u\mapsto u_\sigma:=\sigma^{\frac{3}{2}}u(\sigma\cdot)$. Therefore, we can prove equation \eqref{infimumisnega} holds by first fixing any $u\in\mathcal{H}^1(\mathbb{R}^3)$ with $\vert\vert u\vert\vert_{\mathcal{L}^2(\RRR)}=1$, and then taking $\sigma$ sufficiently small so that  
\begin{align}\int_{\RRR}(\vert\nabla u_\sigma\vert^2-Vu_\sigma^2)dx=\int_{\RRR}\left[\sigma^2\vert\nabla u(x)\vert^2-\sigma^\tau\sum_{k=1}^K\frac{\alpha_k}{\vert x-\sigma r_k\vert^\tau}u^2(x)\right]dx<0.\end{align}
\end{proof}

We will require the following basic energy bound in many of our proofs.  This result is proven in Lemma~6 of Alama-Bronsard-Choksi-Topaloglu~\cite{AlamaBronsardChoksiTopalogluLongRange}; although there it is stated for minimizing sequences, it is clear from the proof that in fact it applies to any function with negative energy:

\begin{Lem}\label{LemaBoundsun}
Assume $V$ satisfies \eqref{condiV}, and $u\in\Hone$ with $\|u\|_{\mathcal{L}^2(\RRR)}^2=M$ and $\EV(u)<0.$ Then, there exists a constant $C_0=C_0(V)>0$ such that 
\begin{equation}\label{unifbd}  \|u\|^2_{\mathcal{H}^1(\RRR)} + \int_{\RRR}\int_{\RRR} {u^2(x)\, u^2(y)\over |x-y|} dx\, dy + \int_{\RRR} V(x) u^2(x)\, dx \le C_0 M.  
\end{equation}
\end{Lem}

\begin{Rmk}Note that boundedness of a sequence $\{u_n\}_{n\in\mathbb{N}}$ in $\Hone$ implies boundedness of the same sequence in $\mathcal{L}^r(\RRR)$, for $2\leq r\leq 6$.\end{Rmk}

The following is stated as part of Proposition~\ref{TSN2}, but its proof only depends on the bounds stated in Lemma~\ref{LemaBoundsun}, and the result will be needed below.

\begin{Prop}\label{MinimizingForZnisMinimizingFore0}
Let $u_n$ minimize $I_{V_n}(M)$, where $V_n$ is as in \eqref{formulaVn}.  Then $\{u_n\}_{n\in\mathbb{N}}$ is a minimizing sequence for $I_{V_{TF}}(M)$.
\end{Prop}

\begin{proof}
Let $\{u_n\}$ be  minimizers  for $I_{V_n}$, $n\in\mathbb{N}$.
First, note that $V_1\ge V_n(x)\ge V_{TF}(x)$ for all $x$, and hence 
$$\mathcal{E}_{V_1}(u_n)\le \mathcal{E}_{V_n}(u_n)=I_{V_n}\le I_{TF}<0$$ 
for all $n$.  Applying Lemma~\ref{LemaBoundsun} with $V=V_1$, the sequence $\{u_n\}_{n\in\NN}$ satisfies the bounds \eqref{unifbd} uniformly in $n\in\NN$. Next, we observe that  $|x|^{-\nu}\in \mathcal{L}^{\frac{3}{2}}_{loc}(\RRR)$ for $0<\nu<1$, and thus
\begin{align}
    Z_n\int_{\RRR} {u_n^2(x)\over |x|^\nu}\, dx
               & \le Z_n \int_{B_1(0)} {u_n^2(x)\over |x|^\nu}dx + Z_n\int_{\RRR\setminus B_1(0)} u_n^2(x)\, dx  \\
            &\le  Z_n \|u_n\|_{\mathcal{L}^6(\RRR)}^2 \, \left\| |x|^{-\nu}\right\|_{\mathcal{L}^{\frac{3}{2}}(B_1(0))} + Z_n M \\
            &\le cZ_n \|\nabla u_n\|_{\mathcal{L}^2(\RRR)}^2 + Z_n M\longrightarrow 0.
\end{align}
 In particular,
$\EVn(u_n)= \ETF(u_n) + o(1),$   and therefore we may conclude,
$$
I_{V_{TF}} \le \liminf_{n\to\infty} \ETF(u_n) = \liminf_{n\to\infty} \EVn(u_n) 
   = \liminf_{n\to\infty} I_{V_n} \le \limsup_{n\to\infty} I_{V_n} \le I_{V_{TF}}.
$$
\end{proof}

\begin{Lem}\label{ExpDecayun} 
Under all hypotheses of Theorem \ref{TheoremLocation}, we have that, up to a subsequence, for all $0<t<\sqrt{-\mu}$, there exists a constant C independent of $n$ with
\begin{align}\begin{split}\label{boundunexp}0<\vert u_n( x )\vert
  + |\nabla u_n(x)| \leq Ce^{-t\sigma_n( x )},\textit{ a.e. in $\RRR$,}\end{split}\end{align}
 where
\begin{align}\begin{split}\sigma_n( x ):=\min_{0\leq i\leq N}\vert x - x _n^{i}\vert,\end{split}\end{align}
and $x_n^i$ are as in the Concentration Theorem~\ref{TheoremSplit}.
\end{Lem}
\begin{proof} 
By Proposition~\ref{properties1} (iii), we may take $u_n(x)>0$ in $\RRR$.
 Alama, Bronsard, Choksi, and Topaloglu~\cite{AlamaBronsardChoksiTopalogluLongRange} proved that
\begin{align}\begin{split}\label{eqnun}-\Delta u_n=\left(\mu_n-\frac{5}{3}c_1u_n^{\frac{4}{3}}+\frac{4}{3}c_2u_n^{\frac{2}{3}}+V_n-u_n^2\star\vert\cdot\vert^{-1}\right)u_n,\end{split}\end{align}
for some $\mu_n<0$. In addition to this, by the final step in the proof of the Concentration Theorem~\ref{TheoremSplit}, the Lagrange multipliers $\mu_n\to\mu$ converge.  Fix $t\in (0,\sqrt{-\mu})$; then, for all $n$ sufficiently large, 
\begin{align}\begin{split} -\Delta u_n+t^2u_n<\left[\frac{1}{2}(t^2+\mu)+\frac{4}{3}c_2u_n^{\frac{2}{3}}\right]u_n.\end{split}\end{align}
Moreover, by Lemma \ref{LemaBoundsun}  and equation \eqref{eqnun} we have that $\{u_n\}_{n\in\mathbb{N}}$ is bounded in $\mathcal{H}^2(\RRR)$, and hence in $\mathcal{L}^\infty(\RRR)$. Therefore, we can make use of Theorem 8.17 by Gilbarg and Trudinger \cite{GilbargTrudinger} to obtain 
\begin{align}\vert\vert u_n\vert\vert_{\mathcal{L}^\infty(B_1(y))}&\leq C\vert\vert u_n\vert\vert_{\mathcal{L}^2(B_2(y))}\leq C\vert\vert u_n\vert\vert_{\mathcal{L}^2(\RRR\setminus \cup B_{R/2}( x _n^{i}))},\quad y\in\RRR\setminus\cup B_{R}(x_n^i),R\gg1,\end{align}
where $C$ is a constant independent of $n, R$ and $y$. By covering $\RRR\setminus\cup B_{R}(x_n^i)$ with balls of radius one centered at points in the same set we obtain
\begin{align}\label{boundlinftyun}\vert\vert u_n\vert\vert_{\mathcal{L}^\infty(\RRR\setminus \cup B_{R}( x _n^{i})))}&\leq C\vert\vert u_n\vert\vert_{\mathcal{L}^2(\RRR\setminus \cup B_{R/2}( x _n^{i}))}.\end{align}

On the other hand, by Proposition \ref{MinimizingForZnisMinimizingFore0}, $\{u_n\}_{n\in\mathbb{N}}$ is a minimizing sequence for $I_{V_{TF}}(M)$, and hence the conclusions of Concentration Theorem \ref{TheoremSplit} hold. In particular, this implies that given $\epsilon>0$, there exists $R_0=R_0(\epsilon)\ge 1$ such that
\begin{align}\vert\vert u^i\vert\vert_{\mathcal{L}^2(\RRR\setminus B_{R/2}( 0 ))}^2<\frac{\epsilon}{N+1},\quad i=0,\dots,N, \ R\geq R_0,\end{align}
and \eqref{ca1}, \eqref{xblow}, Rellich-Kondrakov Theorem, and the decay of all $u^i$ \eqref{expdecay} ensure that, up to a subsequence,
\begin{equation}\label{lanuevacosa}  u_n(\cdot+ x _n^{i})\to u^{i}\textit{ in }\mathcal{L}^2(B_{R/2}( 0 )), \quad i=0,\dots N, \ R\geq R_0.\end{equation}
As a result, 
\begin{align}\lim_{n\to\infty}\vert\vert u_n\vert\vert_{\mathcal{L}^2(\RRR\setminus B_{R/2}( x _n^{i}))}^2&=M-\lim_{n\to\infty}\vert\vert u_n\vert\vert_{\mathcal{L}^2(\cup B_{R/2}( x _n^{i}))}^2\\
&=M-\lim_{n\to\infty}\sum_{i=0}^N\vert\vert u_n(x+x_n^i)\vert\vert_{\mathcal{L}^2(B_{R/2}( 0 ))}^2\\
&=M-\sum_{i=0}^N\vert\vert u^i\vert\vert_{\mathcal{L}^2(B_{R/2}( 0 ))}^2\\
&=\sum_{i=0}^N\vert\vert u^i\vert\vert_{\mathcal{L}^2(\RRR\setminus B_{R/2}( 0 ))}^2<\epsilon.\end{align}
Then, given any $\epsilon>0$, by choosing $R_0=R_0(\epsilon)$ larger if necessary, we have
\begin{align}\limsup_{n\to\infty}\vert\vert u_n\vert\vert_{\mathcal{L}^\infty(\RRR\setminus \cup B_{R}( x _n^{i})))}\leq \epsilon,\quad 
   R\geq R_0,\end{align}
and hence for large enough $n$ and $R$,
\begin{align}\begin{split} -\Delta u_n+t^2u_n<0,\textit{ a.e. in $\RRR\setminus\cup B_R({x}_n^i)$}.\end{split}\end{align}
Next, it is not hard to check that
\begin{align}\begin{split} -\Delta e^{-t\sigma_n}+t^2e^{-t\sigma_n}>0,\textit{ a.e. in $\RRR\setminus\cup B_{R}( x _n^i)$},\end{split}\end{align}
and that there exists $C>0$ so that
\begin{align}u_n\vert_{\partial \cup B_R(x_n^i)}\leq Ce^{-tR}=Ce^{-t\sigma_n( x )}\vert_{\partial\cup B_R(x_n^i)}\end{align}
Thus, 
\begin{align}\begin{split}-\Delta [u_n( x )-Ce^{-t\sigma_n( x )}]+t^2[u_n( x )-Ce^{-t\sigma_n( x )}]<0,\textit{ a.e. in $\RRR\setminus\cup B_{R}( x _n^i)$}.\end{split}\end{align}
At this point, we would like to invoke the maximum principle to assert that $u_n(x)$ is dominated by the supersolution $v(x)=Ce^{-t\sigma_n}$ in the domain $\Omega_n:=\RRR\setminus\cup B_{R}( x _n^i)$.  As the domain is unbounded, this requires some care, but applying Stampacchia's method as in Benguria, Brezis and Lieb \cite[Lemma 8]{BenguriaBrezisLieb} we obtain the desired bound,
$$  0<u_n(x) \le v(x)=Ce^{-t\sigma_n( x )}, \qquad x\in \Omega_n.  $$

The estimate on $|\nabla u_n|$ then follows from standard elliptic estimates; see for instance Theorems~8.22 and 8.32 of
\cite{GilbargTrudinger}.
\end{proof}

At this point we would like to note that the functions $u^i$ decay to zero at infinity, even if $m^0\not\in \mathcal{M}^*_{V_{TF}}$. This follows from Proposition \ref{properties1} (iii) and Theorem 8.17 uniformly by Gilbarg and Trudinger \cite{GilbargTrudinger}, again.



%
\section{Proof of Theorem \ref{TheoremSplitNewtonian}}

The proofs of Theorems \ref{TheoremSplitNewtonian} and \ref{TheoremLocation} both rely on the splitting structure given in the Concentration Theorem~\ref{TheoremSplit}, and on the idea that, when calculating the interaction energy between very widely separated components $u^i(x+x^i_n)$, only the mass $m^i$ and centers $x^i_n$ enter into the computation at first order.  The following simple lemma makes this precise, at least for compactly supported components:

\begin{Lem}\label{localize}  {\bf (a)} \ Let $v^1,v^2\in \Hone$ with compact support, 
supp${}\, v^i\subset B_\rho(\zeta^i)$, $i=1,2$, with $1<\rho<\frac14R$, $R=|\zeta^1-\zeta^2|>0$.  Then,
$$   \left| \int_{B_\rho(\zeta^1)}\int_{B_\rho(\zeta^2)}
    {|v^1(x)|^2 |v^2(y)|^2\over |x-y|}dx\, dy 
       - {\|v^1\|_{\mathcal{L}^2(\RRR)}^2 \|v^2\|_{\mathcal{L}^2(\RRR)}^2  \over |\zeta^1-\zeta^2|} \right| \le
              {4\rho\over R^2} \|v^1\|_{\mathcal{L}^2(\RRR)}^2 \|v^2\|_{\mathcal{L}^2(\RRR)}^2.  $$

{\bf (b)} \  Let $v\in\Hone$ with compact support, 
supp${}\, v\subset B_\rho(\zeta)$, with $1<\rho<\frac14 R=|\zeta|$.  For any $\nu>0$ and fixed vector $ r \in\RRR$ with $0<| r |<\frac14 R$,
$$  \left|  \int_{B_\rho(\zeta)} {|v(x)|^2\over |x- r |^\nu} dx 
         - {\|v\|_{\mathcal{L}^2(\RRR)}^2\over |\zeta|^\nu}\right| \le C_\nu {\rho\over R^{\nu+1}}\|v\|_{\mathcal{L}^2(\RRR)}^2.
$$
\end{Lem}

\begin{proof}  These follow from the pointwise estimates,
\begin{align}\begin{split}\label{jajua1}
\left\vert\frac{1}{\vert \zeta ^1-\zeta ^2\vert}-\frac{1}{\vert x - y \vert}\right\vert
\leq 
{2\rho\over (R-\rho)^2}\le {4\rho\over R^2}, 
\\\label{jajua2}
\left\vert\frac{1}{\vert \zeta \vert^{\theta}}-\frac{1}{\vert x - r \vert^{\theta}}\right\vert\leq \frac{\theta \rho}{(\vert \zeta \vert-\rho-\vert r\vert)^{\theta+1}}
\le C_\theta {\rho\over R^{1+\theta}},\end{split}\end{align}
for all $x\in B_\rho(\zeta^1)$, $y\in B_\rho(\zeta^2)$, and $1<\rho<\frac14 R$.
\end{proof}

Unlike the case of the Gamow liquid drop problem, our components $u^i$ are not of compact support, so we need to resort to truncation.  This will prove effective provided we are in a situation where the minimizers $u^i$ have exponential decay.
To generate localization functions, fix any smooth $\phi:\mathbb{R}\to[0,1]$ for which
\begin{align}\begin{split}\label{asdadasd}\phi\mathbbm{1}_{(-\infty,0]}\equiv1,\quad  \phi\mathbbm{1}_{[1,\infty)}\equiv0,\quad  \vert\vert\phi'\vert\vert_{\mathcal{L}^\infty(\mathbb{R})}\leq2.\end{split}\end{align}
We are now ready to prove Theorem \ref{TheoremSplitNewtonian} and Proposition \ref{TSN2}, on the size of the compact part of minimizing sequences.  The argument for the first Theorem is similar to Lions'  proof of existence of  minimizers \cite{Lions} for TDFW with $M\le\ZT$.

\begin{proof}[Proof of Theorem~\ref{TheoremSplitNewtonian}]
We write the potential $V=V_{TF}+W$, where $W(x)\ge 0$ and $W$ satisfies \eqref{condiV}.  To obtain a contradiction, assume $\{u_n\}_{n\in\mathbb{N}}$ is a minimizing sequence for $I_V(M)$ for which there is splitting (i.e., $N\ge 1$ in Theorem~\ref{TheoremSplit},) but $0<m^0<\ZT$.  We let $u^i$, $m^i=\|u^i\|_{\mathcal{L}^2(\RRR)}^2$, $i=0,\dots,N$ be as given by Theorem~\ref{TheoremSplit}.
Fix $N$ distinct unit vectors ${e^i}\in\RRR$, and for $\rho>1$ define ${q^i}$ by
$$  {q^0}=0, \qquad {q^1}=\rho^2 {e^1}, \qquad {q^i}=\rho^3{e^i}, \  i=2,\dots,N.  $$
Then we define the truncated components,
$$   U^i_\rho(x):= \phi(|x-{q^i}|- \rho + 1)\, u^i(x-{q^i}), \qquad i=0,\dots, N.  $$
That is, each $U^i$ has been truncated to have support in the ball $B_\rho({q^i})$.

As $m^0<\ZT$, by Proposition~\ref{muneg} (ii), $\mu<0$ for all Lagrange multipliers corresponding to $u^i$, $i=0,\dots,N$, and hence the exponential decay estimate \eqref{expdecay} holds for all of them.  Let $\lambda=\frac12\sqrt{-\mu}$ for simplicity.
Then,
$$   m^i_\rho:= \| U^i_\rho\|_{\mathcal{L}^2(\RRR)}^2 \le \int_{B_{\rho}({q^i})} \vert u^i(x)\vert^2dx = m^i - O(e^{-\lambda\rho}),  $$
and
$$  \EV(U^0_\rho)=\EV(u^0) + O(e^{-\lambda\rho}), \qquad
\Ez(U^i_\rho)=\Ez(u^i) + O(e^{-\lambda\rho}), \quad i=1,\dots,N.  $$

Let $w_\rho:= U^0_\rho + \sum_{i=1}^N U^i_\rho.$  As $\|w_\rho\|_{\mathcal{L}^2(\RRR)}^2 < M$, by monotonicity of $I_V(M)$ we have
\begin{align}
I_V(M) &\le I_V(\|w_\rho\|_{\mathcal{L}^2(\RRR)}^2) \le \EV(w_\rho)   \\
   &\le \EV(U^0_\rho) + \sum_{i=1}^N \Ez(U^i_\rho) 
     + \sum_{i,j=0\atop i\neq j}^N \int_{B_\rho({q^i})}\int_{B_\rho({q^j})}
           {|U_\rho^i(x)|^2 |U_\rho^j(y)|^2\over |x-y|} dx\, dy  
           \label{upperbound}  \\
     &\quad - \sum_{i=1}^N \int_{B_\rho({q^i})} [V_{TF}(x)+W(x)]|U^i_\rho(x)|^2\, dx + O(e^{-\lambda\rho}) \\
     &= I_V(m^0) + \sum_{i=1}^N I_0(m^i) 
     + \sum_{i,j=0\atop i\neq j}^N \int_{B_\rho({q^i})}\int_{B_\rho({q^j})}
           {|U_\rho^i(x)|^2 |U_\rho^j(y)|^2\over |x-y|} dx\, dy  
           \label{upperbound2}
             \\
     &\quad - \sum_{i=1}^N \int_{B_\rho({q^i})} [V_{TF}(x)+W(x)]|U^i_\rho(x)|^2\, dx + O(e^{-\lambda\rho}).
\end{align}

Next, we use Lemma~\ref{localize} to evaluate the interaction terms.  Note that $R^{i,j}=| q ^i- q ^j|$ is of order $\rho^2$ when $i=0$, $j=1$, and of order $\rho^3$ otherwise, and $0<m^i-m^i_\rho<O(e^{-\lambda\rho})$.  Thus, we have:
$$   \int_{B_\rho({q^i})}\int_{B_\rho({q^j})}
           {|U_\rho^i(x)|^2 |U_\rho^j(y)|^2\over |x-y|} dx\, dy  =
           \begin{cases}  {\displaystyle{m^0\, m^1\over \rho^2}}+ O(\rho^{-3}), &\text{if $i+j=1$,}\\
                O(\rho^{-3}), &\text{otherwise,}
           \end{cases}
$$
and for $i=1,\dots,N$, $k=1,\dots,K$, we evaluate the interaction with $V_{TF}$ by:
$$
\int_{B_\rho({q^i})} {|U^i_\rho(x)|^2\over |x- r ^k|} dx
   = \begin{cases}
     {\displaystyle{m^1\over \rho^2}} + O(\rho^{-3}), &\text{if $i=1$,}\\
       O(\rho^{-3}), &\text{if $i\ge 2$.}
   \end{cases}
$$
Substituting into \eqref{upperbound2}, and using $W\ge 0$, we obtain the strict subadditivity of $I_V(M)$,
\begin{equation}\label{upperbound3}  I_V(M)- \left[ I_V(m^0) + \sum_{i=1}^N I_0(m^i)\right]\le  { m^0m^1 - \ZT m^1\over \rho^2} + O(\rho^{-3}) <0,  
\end{equation}
for all $\rho$ sufficiently large, since we are assuming $m^0<\ZT$.  However, this  contradicts \eqref{splitting2} in the Concentration Theorem, and  thus $m^0\ge \ZT$, and the Theorem is proven.
\end{proof}

\medskip

\begin{proof}[Proof of Proposition \ref{TSN2}]

In Proposition~\ref{MinimizingForZnisMinimizingFore0} we have already shown that any sequence of minimizers $u_n$ of $I_{V_n}$ forms a minimizing sequence for $I_{V_{TF}}$.    Part (ii) then follows from Theorem \ref{TheoremSplitNewtonian}.
\end{proof}

\section{Proof of Theorem~\ref{TheoremLocation}}

The proof of Theorem~\ref{TheoremLocation} is more intricate than that of Theorem~\ref{TheoremSplitNewtonian}, as it requires us to make a finer estimate of the smaller order terms in the expansion of the energy. 

By (i) of Proposition~\ref{TSN2}, $\{u_n\}_{n\in\mathbb{N}}$ is a minimizing sequence for $I_{V_{TF}}(M)$, and hence the conclusions of Concentration Theorem \ref{TheoremSplit} hold.  We assume there is splitting, that is $N\ge 1$, and let $u^i$, $m^i=\|u^i\|_{\mathcal{L}^2(\RRR)}^2$, and $x^i_n$ (with $x^0_n=0,$) be as in the Concentration Theorem.  By hypothesis, the common value of the Lagrange multipliers of the limit components $u^i$ is negative, $\mu<0$. 

As in the proof of Theorem~\ref{TheoremSplitNewtonian} we construct comparison functions by localization to balls with centers $ q ^i$ spreading to infinity.
However, we have little control on the errors introduced by the passage of $u_n(\cdot-x^i_n)\weak u^i$, and thus we use truncations of the minimizers $u_n$ themselves to make these constructions.
 
Set
\begin{equation}\label{Rndef}   R_n:= \min_{0\le i<j} | x ^i_n- x ^j_n|.  \end{equation}
Consider also a sequence $\rho_n\to\infty$ and translations $\{ q _n^0= 0 \}_{n\in\mathbb{N}},\dots, \{ q _n^N\}_{n\in\mathbb{N}}\subset\RRR$ (all to be chosen later,) with
\begin{equation}\label{papapirs}   1\leq\rho_n\leq\frac{1}{4}\min\{R_n,Q_n\} \quad \text{where} \
      Q_n:=\min_{i< j}\vert  q _n^{i}- q _n^{j}\vert.  
\end{equation}
Using the same cutoff functions $\phi$ defined in \eqref{asdadasd}, we then set
\begin{gather}\chi_{\rho_n}(\cdot):=\phi(\vert\cdot\vert-\rho_n+1),  \\
\label{GHdef}
G_n^i(\cdot):=\chi_{\rho_n}(\cdot- x _n^{i})u_n(\cdot),\quad\textit{ and } \quad H_n^i(\cdot):=G_n^i(\cdot+ x _n^i- q _n^i).\end{gather}
Thus, $G_n^i$ are compactly supported in balls $B_{\rho_n}(x^i_n)$ centered at the $x_n^i$, chosen by the Concentration Theorem, while $H^i_n$ are the same functions but translated to have centers at $ q ^i_n$, which we will choose to create appropriate comparison functions.  Set
\begin{align}m_n^i:=\vert\vert G_n^i\vert\vert^2_{\mathcal{L}^2(\RRR)}=\vert\vert H_n^i\vert\vert^2_{\mathcal{L}^2(\RRR)}.\end{align}

We first confirm that these truncations provide a good approximation to the limit profiles $u^i$, in the $\mathcal{L}^2$ sense.

\begin{Lem}\label{limitmass}
For any $\rho_n$ satisfying \eqref{papapirs},
$$  \lim_{n\to\infty} m_n^i=  m^i=\|u^i\|_{\mathcal{L}^2(\RRR)}^2 .  $$
\end{Lem}

\begin{proof}
First, it is easy to show that $G^i_n(\cdot+ x_n^i)\weak u^i$, $i=0,\dots, N$ weakly in $\Hone$, and in the norm on $\mathcal{L}^2_{loc}(\RRR)$.  As a consequence, 
\begin{equation}\label{liminf}   m^i=\|u_i\|_{\mathcal{L}^2(\RRR)}^2 \le \liminf_{n\to\infty} m_n^i.  
\end{equation}
To obtain the complementary bound, we note that $\sum_{i=0}^N G^i_n(x)\le u_n(x)$ pointwise on $\RRR$, and since the supports of the $G^i_n$ are disjoint we have
\begin{multline}
\limsup_{n\to\infty}  \sum_{i=0}^N m_n^i
   < \| u_n\|_{\mathcal{L}^2(\RRR)}^2 = M = \sum_{i=0}^N m^i 
   \le  \sum_{i=0}^N \liminf_{n\to\infty} m_n^i 
   \le \liminf_{n\to\infty}  \sum_{i=0}^N  m_n^i.
\end{multline}
In particular, the limit $M=\lim_{n\to\infty}\sum_{i=0}^N m_n^i$ exists.  Since individually the terms are bounded below via \eqref{liminf}, we claim that each of the terms $m_n^i \to m^i$, $i=0,\dots,N$.  Indeed, for any $\eps>0$ there exists $K>0$ for which $\sum_{i=0}^N m_n^i < M+\eps$ and for any $i$, $m_n^i\ge m^i -\eps/N$, whenever $n\ge K$. Thus, for each $j$ we have
$$  m_n^j + \sum_{i\neq j} m^i - \eps \le \sum_{i=0}^N m_n^i 
    < \sum_{i=0}^N m^i + \eps,  $$
and so $m_n^j < m^j + 2\eps$, for all $n\ge K$, that is, $\limsup_{n\to\infty} m_n^i\le m^i$, for each $i$, and the claim is proven.
\end{proof}

Since we are assuming $\mu<0$, the exponential decay of $u_n$ away from balls $B_{\rho_n}( x ^i_n)$ allows us to localize the energy $\EVn(u_n)$ with an exponentially small error:

\begin{Lem}\label{LowerBoundd} Let $\rho_n\to\infty$ with $\rho_n\le \frac14 R_n$. Then,
\begin{multline}\mathcal{E}_{V_n}(u_n)\ge \ETF\left(G_n^0\right)+\sum_{i=1}^N\mathcal{E}_{0}\left(G_n^i\right)-Z_n\int_{\RRR}\frac{\vert G_n^0( x )\vert^2}{\vert x \vert^{\nu}}d x  \\
+\sum_{{1\leq i<j}}\frac{m_n^{i}m_n^{j}}{\vert  x _n^{i}- x _n^{j}\vert}+\left(m_n^0-\ZT\right)\sum_{i=1}^N\frac{m_n^i}{\vert x _n^i\vert}-Z_n\sum_{i=1}^N \frac{m_n^{i}}{\vert x _n^{i}\vert^{\nu}} -\eps_{1,n},
\end{multline}
where
\begin{align}\label{epsilondecaysexpo1}\vert \epsilon_{1,n}\vert
\leq C\left({\rho_n\over R_n^2} + {Z_n\rho_n\over R_n^{1+\nu}}+ e^{-\frac{\sqrt{-\mu}}{2}\rho_n}\right),\end{align}
as $n\to\infty$, with $C$ depending on $\{m^i\}$ and $\ZT$ but independent of $\{x^i_n\}$.
\end{Lem}
\begin{proof}By Lemma \ref{ExpDecayun}, for sufficiently large $n$,
\begin{align}\begin{split}\label{expansionss}\left\vert u_n(x)-\sum_{i=0}^NG_n^i(x)\right\vert\leq Ce^{-\frac{\sqrt{-\mu}}{2}\sigma_n(x)},\quad  x\in\Omega_n:=\RRR\setminus\bigcup_{i=0}^N B_{\rho_n}( x ^i_n),\end{split}\end{align}
where $\sigma_n(x)$ is as in Lemma~\ref{ExpDecayun}.
This together with \eqref{boundunexp}, \eqref{eqnun}, $\mu_n\to\mu$, Lemma \ref{LemaBoundsun}, $\vert\vert\nabla\chi_{n,\rho_n}\vert\vert_{\mathcal{L}^\infty(\RRR)}\leq2$, and H\"{o}lder estimates for first derivatives imply 
\begin{align}\begin{split}\label{gradientboundedexpo}
\left\vert\nabla \left(u_n(x)-\sum_{i=0}^NG_n^i(x)\right)\right\vert \leq Ce^{-\frac{\sqrt{-\mu}}{2}\sigma_n(x)}.\end{split}\end{align}
As $G^i_n(x)=u_n(x)$ in $B_{\rho_n}(x^i_n)$, and has support in $B_{\rho_n+1}(x^i_n)$, the contribution to the energy is unchanged in $\bigcup_{i} B_{\rho_n}(x^i_n)$, and is exponentially small in the complementary region, $\Om_n$.  Moreover, the energy density is integrable over $\Om_n$, and of order $\eps_{1,n}=O(e^{-\frac{\sqrt{-\mu}}{2}\rho_n})$.  Hence, we calculate:
\begin{align}\begin{split}\mathcal{E}_{V_n}(u_n)&=\sum_{i=0}^N\mathcal{E}_{V_n}\left(G_n^i\right)+\sum_{i<j}\int_{B_{\rho_n}( x _n^i)}\int_{B_{\rho_n}( x _n^j)}\frac{\vert G_n^i( x )\vert^2\vert G_n^j( y )\vert^2}{\vert x - y \vert}d x d y +\epsilon_{1,n} \\
&=\ETF\left(G_n^0\right)-Z_n\int_{B_{\rho_n}( 0 )}\frac{\vert G_n^0( x )\vert^2}{\vert x \vert^{\nu}}d x \\
&\ \ \ +\sum_{i=1}^N\left[\mathcal{E}_{0}\left(G_n^i\right)-\int_{\RRR}V( x )\vert G_n^i( x )\vert^2d x -Z_n\int_{B_{\rho_n}( x _n^i)}\frac{\vert G_n^i( x )\vert^2}{\vert x \vert^{\nu}}d x \right]\\
&\ \ \ +\sum_{i<j}\int_{B_{\rho_n}( x _n^i)}\int_{B_{\rho_n}( x _n^j)}\frac{\vert G_n^i( x )\vert^2\vert G_n^j( y )\vert^2}{\vert x - y \vert}d x d y +\epsilon_{1,n}\\
&=\ETF\left(G_n^0\right)+\sum_{i=1}^N\mathcal{E}_{0}\left(G_n^i\right)-Z_n\int_{B_{\rho_n}( 0 )}\frac{\vert G_n^0( x )\vert^2}{\vert x \vert^{\nu}}d x  \\
&\ \ \ +\sum_{i<j}\int_{B_{\rho_n}( x _n^i)}\int_{B_{\rho_n}( x _n^j)}\frac{\vert G_n^i( x )\vert^2\vert G_n^j( y )\vert^2}{\vert x - y \vert}d x d y \\
&\ \ \ -\sum_{i=1}^N\sum_{k=1}^K\alpha_k\int_{\RRR}\frac{\vert G_n^i( x )\vert^2}{\vert x - r _k\vert}d x -Z_n\sum_{i=1}^N\int_{B_{\rho_n}( x _n^i)}\frac{\vert G_n^i( x )\vert^2}{\vert x \vert^{\nu}}d x +\epsilon_{1,n}.  \label{enidentity}\end{split}
\end{align}

Now, we apply Lemma~\ref{localize} to evaluate the interaction terms.  In this way we have:
\begin{gather}
\int_{B_{\rho_n}( x _n^i)}\int_{B_{\rho_n}( x _n^j)}\frac{\vert G_n^i( x )\vert^2\vert G_n^j( y )\vert^2}{\vert x - y \vert}d x d y 
\ge {m^i_n m^j_n\over | x ^i_n- x ^j_n|} 
- 4m^i_n m^j_n{\rho_n\over R_n^2}, \\
\int_{\RRR}\frac{\vert G_n^i( x )\vert^2}{\vert x - r _k\vert}d x  
\le  {m^i_n \over |  x _n^i|} + C_1m^i_n {\rho_n\over R_n^2}, \\
\int_{B_{\rho_n}( x _n^i)}\frac{\vert G_n^i( x )\vert^2}{\vert x \vert^{\nu}}d x  \le {m^i_n\over | x ^i_n|^\nu} +
    C_\nu m^i_n {\rho_n\over R_n^{\nu+1}}.
\end{gather}
By substituting these estimates into \eqref{enidentity} we arrive at the desired lower bound.
\end{proof}

Next we create an upper bound estimate on the minimum energy by moving the localized components $H^i_n$ (which are simply translates of $G^i_n$,) to study the role of the $x^i_n$.  That is, we consider a trial function $w_n=\sum_{i=0}^N H^i_n$, which has the same localized components as $u_n$ but with centers $q^i_n$.  The advantage of this over the upper bound constructed for the proof of Theorem~\ref{TheoremSplitNewtonian} is that the terms of order $O(1)$ will exactly match those in the lower bound given by Lemma~\ref{LowerBoundd}.

\begin{Lem}\label{UpperBound}Let $\{\rho_n\}_{n\in\mathbb{N}}\subset(1,\infty)$ and $\{ q _n^{0}=0\}_{n\in\mathbb{N}},\dots,\{ q _n^{N}\}_{n\in\mathbb{N}}\subset\RRR$ satisfy \eqref{papapirs}. Then,
\begin{multline}\mathcal{E}_{V_n}(u_n)<\ETF\left(G_n^0\right)+\sum_{i=1}^N\mathcal{E}_{0}\left(G_n^i\right)-Z_n\int_{\RRR}\frac{\vert G_n^0( x )\vert^2}{\vert x \vert^{\nu}}d x \\ +\sum_{1\leq i<j\le N}\frac{m_n^{i}m_n^{j}}{\vert  q _n^{i}- q _n^{j}\vert}+\left(m_n^0-\ZT\right)\sum_{i=1}^N\frac{m_n^i}{\vert q _n^i\vert}-Z_n\sum_{i=1}^N \frac{m_n^{i}}{\vert q _n^{i}\vert^{\nu}} + \eps_{2,n},
\end{multline}
where
\begin{align}\label{epsilondecaysexpo2}\vert \eps_{2,n}\vert
\leq C\left({\rho_n\over Q_n^2} + {Z_n\rho_n\over Q_n^{1+\nu}}+ e^{-\frac{\sqrt{-\mu}}{2}\rho_n}\right),\end{align}
as $n\to\infty$, with $C$ depending on $\{m^i\}$ and $\ZT$ but independent of $\{q^i_n\}$.
\end{Lem}
\begin{proof}Set
\begin{align}\begin{split}w_{n}:=\sum_{i=0}^NH_n^i.\end{split}\end{align}
As $0\le w_n(x)\le u_n(x)$ for all $x\in\RRR$, $\|w_n\|_{\mathcal{L}^2(\RRR)}^2<\|u_n\|_{\mathcal{L}^2(\RRR)}^2$.  By the monotonicity of $I_{V_n}(M)$ (Proposition~\ref{properties1} (i),) 
\begin{align}\begin{split}\label{EunDomEwn}\mathcal{E}_{V_n}(u_n)=I_{V_n}(M)<I_{V_n}\left(\vert\vert w_n\vert\vert_{\mathcal{L}^2(\RRR)}^2\right)\leq \mathcal{E}_{V_n}(w_n).\end{split}\end{align}
Using the support properties of $H^i_n$ and recognizing $\ETF(H^0_n)=\ETF(G^0_n)$, $\Ez(H^i_n)=\Ez(G^i_n)$, we expand as in the proof of Lemma~\ref{LowerBoundd} to obtain the desired upper bound.
\end{proof}

By matching the lower bound from Lemma~\ref{LowerBoundd} with the upper bound from Lemma~\ref{UpperBound}, we conclude for any choice of $\rho_n$, $\{ q ^i_n\}$ satisfying \eqref{papapirs}, we have the following bound satisfied by the translations $\{ x ^i_n\}$:
\begin{multline}\label{LowerUpperBound}\sum_{ 1\leq i<j\le N}\frac{m_n^{i}m_n^{j}}{\vert  x _n^{i}- x _n^{j}\vert}+\left(m_n^0-\ZT\right)\sum_{i=1}^N\frac{m_n^i}{\vert x _n^i\vert}-Z_n\sum_{i=1}^N \frac{m_n^{i}}{\vert x _n^{i}\vert^{\nu}} \\
\le \sum_{ 1\leq i<j\le N}\frac{m_n^{i}m_n^{j}}{\vert  q _n^{i}- q _n^{j}\vert}+\left(m_n^0-\ZT\right)\sum_{i=1}^N\frac{m_n^i}{\vert q _n^i\vert}-Z_n\sum_{i=1}^N\frac{m_n^{i}}{\vert q _n^{i}\vert^{\nu}} +\eps_{1,n} +\eps_{2,n},
\end{multline}

where $\eps_{1,n},\eps_{2,n}$ are defined in the statements of the Lemmas~\ref{LowerBoundd} and~\ref{UpperBound}.

In what follows we exploit the freedom we have of choosing vectors $ q _n^i$ and radii $\rho_n$ to prove Theorem \ref{TheoremLocation}.  
First we must find the correct scale for the diverging centers $\{x^i_n\}$.  We define
\begin{equation}\label{Rbar0}
\Rnz:=\min_{1\leq i\leq N}|x^i_n|,
\end{equation}
and for $N\ge 2$,
\begin{equation}\label{Rbar}
\Rnb:= \min_{ 1\leq i<j\le N} |x^i_n-x^j_n|.
\end{equation}
By Concentration Theorem~\ref{TheoremSplit}, each diverges to infinity, and moreover $R_n=\min\{\Rnb,\Rnz\}$ (see \eqref{Rndef}.)  By passing to a subsequence and reordering the components if necessary, we may assume that the first diverging center is the closest:
$$    |x_n^1|= \Rnz, \quad\forall n\in\NN.  $$

\begin{Lem}\label{LemmaG} (a) \ If $m^0>\mathcal{Z}$, then
\begin{equation}\label{GG} \liminf_{n\to\infty} R_n Z_n^{\frac{1}{1-\nu}}>0.\end{equation}
(b) \ If $N\ge 2$ and $\displaystyle\limsup_{n\to\infty} {\Rnz\over \Rnb}>0$, then 
there exists a subsequence for which \eqref{GG} holds.
\newline
(c) \  If $N\ge 2$ and ${\Rnz\over \Rnb}\to 0$ then 
$$ \liminf_{n\to\infty} \Rnb Z_n^{\frac{1}{1-\nu}}>0.$$
\end{Lem}

\begin{proof}  First assume $m^0>\mathcal{Z}$.  
To derive a contradiction, assume that (along some subsequence) $R_n Z_n^{\frac{1}{1-\nu}}\to 0$.  Choose $ q ^i_n=R_n  p ^i$, for distinct fixed vectors $ p ^i$, $i=1,\dots,N$, and $ p^0= 0 $.  We also denote by
$ y ^i_n=R_n^{-1} x ^i_n$.  By the definition of $R_n$, we have $|y^i_n|\ge 1$ for all $i=1,\dots, N$, and 
$| y ^i_n- y ^j_n|\ge 1$ for all $0\le i<j\le N$.  Extracting a further subsequence if necessary, we may assume that either
\begin{equation}\label{closest}  |y^1_n|= 1, \ \text{or, \quad there exists $i_0,j_0\neq 0$ for which} \ | y ^{i_0}_n- y ^{j_0}_n|=1, \forall n\in\mathbb{N}.
\end{equation}
   Set $\rho_n=\sqrt{R_n}$, and so \eqref{papapirs} is satisfied for these choices, and in fact
$R_n\eps_{1,n},R_n\eps_{2,n}\to 0$, where $\eps_{1,n},\eps_{2,n}$ are the remainder terms defined in Lemmas~\ref{LowerBoundd} and \ref{UpperBound}.

We multiply \eqref{LowerUpperBound} by $R_n$ to obtain:
\begin{align}\begin{split}\label{GGG}&\sum_{1\leq i<j\le N}\frac{m_n^{i}m_n^{j}}{\vert  y _n^{i}- y _n^{j}\vert}+\left(m_n^0-\ZT\right)\sum_{i=1}^N\frac{m_n^i}{\vert y _n^i\vert} \\
&\qquad \le \sum_{ 1\leq i<j\le N}\frac{m_n^{i}m_n^{j}}{\vert  p ^{i}- p ^{j}\vert}+\left(m_n^0-\ZT\right)\sum_{i=1}^N\frac{m_n^i}{\vert p ^i\vert} \\
&\qquad\quad -Z_nR_n^{1-\nu}\sum_{i=1}^N\frac{m_n^{i}}{\vert p ^{i}\vert^{\nu}} +Z_nR_n^{1-\nu}\sum_{i=1}^N \frac{m_n^{i}}{\vert y _n^{i}\vert^{\nu}}+R_n\eps_{1,n} +R_n\eps_{2,n} \\
&\qquad \le \sum_{ 1\leq i<j\le N}\frac{m_n^{i}m_n^{j}}{\vert  p ^{i}- p ^{j}\vert}+\left(m_n^0-\ZT\right)\sum_{i=1}^N\frac{m_n^i}{\vert p ^i\vert} + o(1),
\end{split}\end{align}
as $Z_nR^{1-\nu}_n\to 0$ by the contradiction hypothesis.  Assuming that $| y ^{i_0}_n- y ^{j_0}_n|=1$ is chosen in \eqref{closest}, we then obtain
$$  m^{i_0}_n\, m^{j_0}_n \le \sum_{ 1\leq i<j\le N}\frac{m_n^{i}m_n^{j}}{\vert  p ^{i}- p ^{j}\vert}+\left(m_n^0-\ZT\right)\sum_{i=1}^N\frac{m_n^i}{\vert p ^i\vert} +o(1),
$$
which holds for all $n$ and any choice of vectors $ p ^i$.  Since $m^i_n\to m^i>0$, we obtain a contradiction by choosing $|p^i-p^j|$ (with $0\le i<j\le N$) sufficiently large.  
If the choice in \eqref{closest} yields $|y^1_n|=1$, we instead have 
$$  (m^{0}_n-\ZT)\, m^{1}_n \le \sum_{{1\leq i<j}}\frac{m_n^{i}m_n^{j}}{\vert  p ^{i}- p ^{j}\vert}+\left(m_n^0-\ZT\right)\sum_{i=1}^N\frac{m_n^i}{\vert p ^i\vert} +o(1).
$$
As we are assuming $m^0=\lim_{n\to \infty} m^0_n>\ZT$ we arrive at the same contradiction as above, choosing $|p^i-p^j|$ (with $1\le i<j\le N$) sufficiently large.
This completes the proof of (a).

For (b), we assume $N\ge 2$ and there exists a subsequence and $r>0$ for which 
 $\Rnz\ge r \Rnb$, but  $R_n Z_n^{\frac{1}{1-\nu}}\to 0$.   Recall that 
 $R_n=\min\{\Rnz,\Rnb\},$
  and so 
  $$\min\{r,1\}\Rnb\le R_n\le \Rnb,$$
   and so each of $\Rnz,\Rnb,R_n$ is of the same order of magnitude.
As in part (a), let $y_n^i=R_n^{-1}x_n^i$, $q_n^i=R_n p^i$, and choose $i_0,j_0$ for which $|x^{i_0}-x^{j_0}|=\Rnb$.  Note that 
$$  |y_n^{i_0}-y_n^{j_0}|^{-1} = {R_n \over \Rnb}\ge \min\{1, r\}. $$
    Again, multiply \eqref{LowerUpperBound} by $R_n$, and pass to the limit as in \eqref{GGG} to obtain:
\begin{equation}\label{ruth} \min\{1, r\} m^{i_0}_n\, m^{j_0}_n +\left(m_n^0-\ZT\right)\sum_{i=1}^N\frac{m_n^i}{\vert y_n ^i\vert}\le \sum_{0<i<j}\frac{m_n^{i}m_n^{j}}{\vert  p ^{i}- p ^{j}\vert}+\left(m_n^0-\ZT\right)\sum_{i=1}^N\frac{m_n^i}{\vert p ^i\vert} +o(1),
\end{equation}
for all $n$ and any choice of vectors $ p ^i$.  Since $m^i_n\to m^i>0$ and $m^0\geq \ZT$ by Theorem~\ref{TheoremSplitNewtonian}, we obtain a contradiction by choosing vectors $p^i$ with $|p^i-p^j|$ sufficiently large.

To prove (c) assume ${\Rnz\over \Rnb}\to 0$, and suppose (for a contradiction) that $\Rnb Z_n^{{1\over 1-\nu}}\to 0$.
First, we note that 
\begin{align}\vert x_n^i\vert&=\vert x_n^i-x_n^1+x_n^1\vert\geq\vert x_n^i-x_n^1\vert-\vert x_n^1\vert\geq \Rnb-\Rnz \geq \frac{1}{2}\Rnb \gg \vert x_n^1\vert,\quad i\geq2,n\gg1.\end{align}
and so only one of the centers is much closer to the origin than the others, $|x_n^1|\ll \Rnb\le |x_n^i|$, for all $i=2,\dots,N$.

Choose cut-off radii $\rho_n$ in Lemmas~\ref{LowerBoundd} and \ref{UpperBound} with $\Rnz\ll \rho_n \ll \Rnb$; for instance, $\rnb = \sqrt{\Rnz\Rnb}$.  Notice that the ball $B_{\rho_n}(0)$ now includes both $x_n^0=0$ and $x_n^1$.  In particular, when defining the disjoint components $G_n^i, H_n^i$ with $\Rnb$ and $\rnb$, we no longer have a component with $i=1$, but the $i=0$ piece accounts for the mass concentrating both at the origin and at $x_n^1$.  In particular, we will have,
\begin{equation} \label{combined}  \| G_n^0\|^2_{\mathcal{L}^2(\RRR)} = \| H_n^0\|^2_{\mathcal{L}^2(\RRR)} = m^0+m^1 + o(1) >\ZT.  
\end{equation}
In this way, we return to the same situation as in part (a), but where $\Rnb$ replaces $R_n$ as the decisive length scale.  As in (a), we choose distinct vectors $q^0=0$ and  $q^i$, $i=2,\dots, N$, and set $p^i_n:=\Rnb q^i$, and (as before) $y_n^i=x_n^i/\Rnb$.  Modulo a subsequence, either there is a pair with 
$|y_n^{i_0}-y_n^{j_0}|=1$, ($i_0,j_0\ge 2$,) or $i_0\ge 2$ with $|y_n^{i_0}|=1, \ \forall n$.   Then we multiply \eqref{LowerUpperBound} by $\Rnb$, to obtain:
\begin{multline}\label{LowerUpper2}\sum_{{2\leq i<j}}\frac{m_n^{i}m_n^{j}}{\vert  y _n^{i}- y_n^{j}\vert}+\left(m^0+m^1-\ZT+o(1)\right)\sum_{i=2}^N\frac{m_n^i}{\vert y _n^i\vert}-Z_n\Rnb^{1-\nu}\sum_{i=2}^N \frac{m_n^{i}}{\vert y _n^{i}\vert^{\nu}} \\
\le\sum_{{2\leq i<j}}\frac{m_n^{i}m_n^{j}}{\vert  p^{i}- p^{j}\vert}+\left(m^0+m^1-\ZT+o(1)\right)\sum_{i=2}^N\frac{m_n^i}{\vert p^i\vert}-Z_n \Rnb^{1-\nu}\sum_{i=2}^N\frac{m_n^{i}}{\vert p^{i}\vert^{\nu}} +\Rnb\overline{\eps}_{1,n} +\Rnb\overline{\eps}_{2,n},
\end{multline}
where $\overline{\eps}_{1,n}$ and $\overline{\eps}_{2,n}$ satisfy \eqref{epsilondecaysexpo1} and \eqref{epsilondecaysexpo2}, for $\Rnb,\rnb$ replacing $R_n,\rho_n$. In particular, $\Rnb \overline{\eps}_{1,n}$,$\Rnb \overline{\eps}_{2,n}\to 0$.  Employing the contradiction hypothesis $\Rnb Z_n^{{1\over 1-\nu}}\to 0$,
and the choice of $\Rnb,\rnb$ we deduce that (in the case $|y_n^{i_0}-y_n^{j_0}|=1$,)
$$
  m^{i_0}_n\, m^{j_0}_n \le \sum_{{2\leq i<j}}\frac{m_n^{i}m_n^{j}}{\vert  p ^{i}- p ^{j}\vert}+\left(m^0+m^1-\ZT+o(1)\right)\sum_{i=2}^N\frac{m_n^i}{\vert p ^i\vert} +o(1)   
$$
or (in the case $|y_n^{i_0}|=1$,)
$$
 (m^0+m^1-\ZT) m^{i_0} \le \sum_{{2\leq i<j}}\frac{m_n^{i}m_n^{j}}{\vert  p ^{i}- p ^{j}\vert}+\left(m^0+m^1-\ZT+o(1)\right)\sum_{i=2}^N\frac{m_n^i}{\vert p ^i\vert} +o(1) 
$$
In either case, we then arrive at the same contradiction as in (a), by choosing $|p^i-p^j|$ large enough, $i\neq j$.
\end{proof}

We now prove the main theorem on the convergence of concentration points at the scale  $R_n=O(Z_n^{-{1\over 1-\nu}})$.

\begin{proof}[Proof of Theorem \ref{TheoremLocation}]
Let $u_n$ attain the minimum in $I_{V_n}$ for each $n\to\infty$.  Applying the Concentration Theorem~\ref{TheoremSplit}, we obtain values of $N$, masses $m^{0},\dots,m^{N}$, and translations $\{ x ^i_n\}$.  

For part (i), we assume $m^0\in\mathcal{M}_{V_{TF}}^*$ and $m^0>\ZT$. 
For any choice of $N$ and masses $m^{0},\dots,m^{N}$ with $m^0>\ZT$, all minimizing sequences for $F_{N,(m^{0},\dots,m^{N})}( w ^1,\dots, w ^{N})$  on $\Sigma_N$ are convergent.  This follows by exactly the same argument as in the proof of Proposition 8 of Alama, Bronsard, Choksi, and Topaloglu~\cite{AlamaBronsardChoksiTopalogluDroplet}.  
Let $( a ^1,\dots, a ^{N})\in\Sigma_N$ be such a minimizer,
\begin{align}\begin{split}F_{N,(m^{0},\dots,m^{N})}( a ^1,\dots, a ^{N})=\min_{( w ^{1},\dots, w ^{N})\in\Sigma_N}F_{N,(m^{0},\dots,m^{N})}( w^1,\dots, w ^{N})<0.\end{split}\end{align}
Define $ \xi _n^i:=Z_n^{\frac{1}{1-\nu}} x _n^i$.  By Lemma~\ref{LemmaG}, $|\xi^i_n|, |\xi_n^i-\xi_n^j|\ge c>0$ are bounded below, for each $i=1,\dots,N$ and $j\neq i$.

Set 
\begin{align}\rho_n:=Z_n^{-\frac{1}{2(1-\nu)}},\textit{ and } q _n^{i}:=Z_n^{-\frac{1}{1-\nu}} a ^{i}.\end{align} 
Then, by the previous Lemma, up to a subsequence,
\begin{align}\begin{split}\label{laraca}1\leq\rho_n\leq\frac{1}{4}\min_{i< j}\{\vert  q _n^{i}- q _n^{j}\vert,R_n\},\end{split}\end{align}
so that equation \eqref{LowerUpperBound} holds, and 
\begin{multline}\sum_{{1\leq i<j}}\frac{m_n^{i}m_n^{j}}{\vert  \xi_n^i- \xi_n^{j}\vert}+\left(m_n^0-\ZT\right)\sum_{i=1}^N\frac{m_n^i}{\vert \xi_n^{i}\vert}-  \sum_{i=1}^N \frac{m_n^{i}}{\vert \xi_n^{i}\vert^{\nu}}\\
\le\sum_{{1\leq i<j}}\frac{m_n^{i}m_n^{j}}{\vert  a^{i}- a^{j}\vert}+\left(m_n^0-\ZT\right)\sum_{i=1}^N\frac{m_n^i}{\vert a^{i}\vert}-  \sum_{i=1}^N\frac{m_n^{i}}{\vert a^{i}\vert^{\nu}}+Z_n^{-\frac{1}{1-\nu}}(\epsilon_n+\hat{\epsilon}_n), \label{x2}
\end{multline}
where $\epsilon_n$ and $\hat\epsilon_n$ satisfy \eqref{epsilondecaysexpo1} and \eqref{epsilondecaysexpo2}, correspondingly. In particular, $Z_n^{-\frac{1}{1-\nu}}\epsilon_n,Z_n^{-\frac{1}{1-\nu}}\hat\epsilon_n\to0$.

In addition to this, by Lemma \ref{LemmaG}. 
\begin{align}\label{x1} \liminf_{n\to\infty}\vert \xi _n^i- \xi _n^j\vert\geq\liminf_{n\to\infty}Z_n^{\frac{1}{1-\nu}}R_n>0.\end{align}
By Lemma \ref{limitmass}, $\lim_{n\to\infty} m_n^i=m^i$, and hence applying \eqref{x2} and \eqref{x1} we obtain 
\begin{align}
\limsup_{n\to\infty}F_{N,(m^0,\dots,m^{N})}&( \xi _n^{1},\dots, \xi _n^{N})\\
&=\limsup_{n\to\infty}\left[\sum_{ 1\leq i<j\le N}\frac{m_n^{i}m_n^{j}}{\vert   \xi _n^i-  \xi _n^j\vert}+\left(m_n^0-\ZT\right)\sum_{i=1}^N\frac{m_n^i}{\vert   \xi _n^i\vert}-  \sum_{i=1}^N \frac{m_n^{i}}{\vert  \xi _n^i\vert^{\nu}}\right]\\
&\leq \limsup_{n\to\infty}\left[\sum_{ 1\leq i<j\le N}\frac{m_n^{i}m_n^{j}}{\vert   a ^i-  a ^j\vert}+\left(m_n^0-\ZT\right)\sum_{i=1}^N\frac{m_n^i}{\vert  a ^i\vert}-  \sum_{i=1}^N\frac{m_n^{i}}{\vert  a ^i\vert^{\nu}}\right]\\
&=F_{N,(m^{0},\dots,m^{N})}( a ^{1},\dots, a ^{N})\\
&=\min_{( w ^{1},\dots, w ^{N})\in\Sigma_N}F_{N,(m^{0},\dots,m^{N})}( w ^{1},\dots, w ^{N}).
\label{x12}
\end{align}
Therefore,  $\{(\xi^1_n,\dots,\xi_n^N)\}_{n\in\mathbb{N}}$  is a minimizing sequence for $F_{N,(m^{0},\dots,m^{N})}$ in $\Sigma_N$, and by \cite[Proposition 8]{AlamaBronsardChoksiTopalogluDroplet}, $\xi^i_n=x^i_nZ^{{1\over 1-\nu}}\to y^i$, $i=0,\dots,N$, with  $(y^1,\dots,y^N)$ a minimizing configuration for $F_{N,(m^{0},\dots,m^{N})}$. This completes the proof in case $m^0>\ZT$.

\medskip

Now consider (ii), for which $m^0=\ZT$.  We first show that $|x^1_n|\ll Z_n^{{1\over \nu-1}}$.  Indeed, assume the contrary that, up to a subsequence, $|x^1_n|\, Z_n^{{1\over 1-\nu}}\ge c>0$ for all $n$.   In case $N\ge 2$, by part (b) of Lemma~\ref{LemmaG}, then $R_n Z_n^{{1\over \nu-1}}\ge c'>0$. 
As in the proof of (i), define $\xi^i_n:= x^i_nZ_n^{{1\over 1-\nu}}$; then $\vert\xi^i_n\vert\ge c$, $i=1,\dots,N$, is bounded below.  We also fix any distinct points $p^1,\dots,p^N \in\R^3\setminus\{0\}$ and $q^i_n=p^i Z_n^{{1\over\nu-1}}$.  

We now proceed as above, arriving at \eqref{x2}.  Note that the inequality \eqref{x2} holds for any $N\ge 1$.  In fact, if $N=1$ the inequality simplifies significantly: the double sums are not present, and only the $i=1$ terms remain.
Passing to the limit as in \eqref{x12}, and recalling $m^0_n\to \ZT$, we then have
\begin{equation}\label{y1} \limsup_{n\to\infty}F_{N,(\ZT, m^1,\dots,m^{N})}(\xi_n^1,\dots,\xi_n^N) \le 
F_{N,(\ZT, m^1,\dots,m^{N})}( p^{1},\dots, p^{N}),\end{equation}
for any choice of distinct nonzero vectors $p^1,\dots,p^N$ in $\R^3$.
Now, as the $\xi^i_n$ are bounded below, the left hand side of the above inequality is finite.  However, the function $F_{N,(\ZT, m^1,\dots,m^{N})}$ is unbounded below, and thus we may choose $p^1,\dots,p^N$ so as to contradict the inequality.  We conclude that $|x_n^1|\ll Z_n^{{1\over \nu-1}}$.

Lastly, for $m^0=\ZT$ and $N\ge 2$ we prove the asymptotic distribution of the concentration centers.
For this, we return to the definitions of $\Rnb,\rnb$ in the proof of Lemma~\ref{LemmaG} (c) above, in which we proved that $\Rnb\ge c Z_n^{{1\over \nu-1}}$.  We recall that the components $G_n^0,H_n^0$ defined in \eqref{GHdef} (but using $\rnb$ in the cut-off $\chi_{\rnb}$,) will enclose neighborhoods of both $x_n^0=0$ and $x_n^1$, and hence their masses combine in $G_n^0,H_n^0$, as in \eqref{combined}.
By the same arguments as in 
\cite[Proposition 8]{AlamaBronsardChoksiTopalogluLongRange}, all minimizing sequences of 
the interaction energy $\overline{F}_{N,(m^{1},m^2,\dots,m^{N})}$ converge to a minimizer 
$(y^2,\dots,y^N)\in \overline{\Sigma}_N$.  Define $q_n^0=0$ and $q^i_n= y^i Z_n^{{1\over \nu-1}}$, $i=2,\dots,N$.
Applying \eqref{LowerUpperBound} with these choices, we have:
\begin{multline}\label{LowerUpperBoundSecondCase}\sum_{{2\leq i<j}}\frac{m_n^{i}m_n^{j}}{\vert  x _n^{i}- x _n^{j}\vert}+\left(m^0+m^1-\ZT +o(1) \right)\sum_{i=2}^N\frac{m_n^i}{\vert x _n^i\vert}-Z_n\sum_{i=2}^N \frac{m_n^{i}}{\vert x _n^{i}\vert^{\nu}} \\
\le\sum_{{2\leq i<j}}\frac{m_n^{i}m_n^{j}}{\vert  q _n^{i}- q _n^{j}\vert}+\left(m^0+m^1-\ZT +o(1) \right)\sum_{i=2}^N\frac{m_n^i}{\vert q _n^i\vert}-Z_n\sum_{i=2}^N\frac{m_n^{i}}{\vert q _n^{i}\vert^{\nu}} +\overline{\eps}_{1,n} +\overline{\eps}_{2,n},
\end{multline}
with (as in part (i)) $Z_n^{{1\over \nu-1}}\overline{\eps}_{1,n},Z_n^{{1\over \nu-1}}\overline{\eps}_{2,n}\to 0$.  Multiplying the above inequality by $Z_n^{{1\over \nu-1}}$, we pass to the limit and obtain an inequality for $\overline{F}_{N,(m^{1},m^2,\dots,m^{N})}$,
$$  \limsup_{n\to\infty} \overline{F}_{N,(m^{1},m^2,\dots,m^{N})} (\xi_n^2,\cdots,\xi_n^N) 
  \le \overline{F}_{N,(m^{1},m^2,\dots,m^{N})}(y^2,\cdots,y^N).
$$
Again, the renormalized centers $(\xi^2_n,\dots,\xi^N_n)$ give a minimizing sequence for $\overline{F}_{N,(m^{1},m^2,\dots,m^{N})}$ and must converge to a minimizer. 
This completes the proof of Theorem~\ref{TheoremLocation}.
\end{proof}


\section*{Acknowledgement}
We acknowledge the support of the Natural Sciences and Engineering Research Council of Canada (NSERC) Discovery Grants program.

\appendix
\section{Appendix: Proof of the Concentration Theorem}

In this section we prove the Concentration Theorem~\ref{TheoremSplit}.  The use of Concentration-Compactness techniques in Thomas-Fermi-type problems goes back at least to Lions~\cite{Lions}, for whom these problems were an important motivation for the development of the general theory.  The result of the Concentration Theorem~\ref{TheoremSplit} is essentially contained in Lions\cite{Lions}, although not as a single Theorem and with many details left to the reader.  Since we make heavy use of the decomposition into minimizers in the main results of the paper, we provide a more complete proof here (with specific references to steps appearing in other articles.)


\begin{proof}[Proof of Theorem \ref{TheoremSplit}]
We first present the proof with $V\not\equiv 0$; the case $V\equiv 0$ requires only a simple modification.
Let $\{u_n\}_{n\in\mathbb{N}}$ be a minimizing sequence for $\EV$ with $\|u_n\|_{\mathcal{L}^2(\RRR)}^2=M$.  
 Since $\mathcal{E}_V$ is coercive, $\{u_n\}_{n\in\mathbb{N}}$ is bounded in $\mathcal{H}^1(\RRR)$.  Hence, there exists $u^0\in\Hone$ and a subsequence for which
$u_n\rightharpoonup u^0$ weakly in $\Hone$.  At this point it is not clear if $u^0$ is nontrivial; this will be shown later. Let $m^0:=\|u^0\|_{\mathcal{L}^2(\RRR)}^2$ and $x^0_n=0$.  If $m^0=M$, then  the sequence converges strongly in $\mathcal{L}^2$, and $u^0$ minimizes $I_V(M)$, and the procedure terminates, with $N=0$.   

If instead $m^0<M$, we define the remainder $u^0_n(x):= u_n(x)- u^0(x+x^0_n)$.  Note that  by weak convergence, $\|u^0_n\|_{\mathcal{L}^2(\RRR)}^2 \to M-m^0$, and 
by weak convergence and the Brezis-Lieb Lemma, the energy decouples in the limit (see \cite{NamVBosch,Lieb}),
$$   \lim_{n\to\infty} \left[ \EV(u_n)- \EV(u^0)-\Ez(u^0_n)\right]=0, $$
and thus
$$   I_V(M) =\lim_{n\to\infty} \EV(u_n)=  \EV(u^0) + \lim_{n\to\infty} \Ez(u^0_n) \ge I_V(m^0) + I_0(M-m^0).  $$
By the binding inequality \eqref{binding}, we have
\begin{equation}\label{firststep}
  I_V(m^0) + I_0(M-m^0) \ge I_V(M)\ge \EV(u^0) + \lim_{n\to\infty} \Ez(u^0_n) \ge I_V(m^0) + I_0(M-m^0), 
  \end{equation}
and hence we obtain equality of each expression,
$$  \EV(u^0)=I_V(m^0) \quad\text{and}
   \quad \lim_{n\to\infty} \Ez(u^0_n)= I_0(M-m^0),  $$
that is, the remainder sequence $\{u^0_n\}_{n\in\mathbb{N}}$ is a minimizing sequence for $I_0(M-m^0)$.

We next consider the residual sequence $\{u^0_n\}_{n\in\mathbb{N}}$ and show it concentrates after translation.  First, we must eliminate the possibility of ``vanishing'' in the Concentration-Compactness framework \cite{LionsConcentrationPart1}.  To this end, for any bounded sequence we define  (as in Nam-van den Bosch \cite{NamVBosch},)
$$   \omega(\{v_n\}):= \sup\left\{ \|v\|^2 \ \bigl| \  \exists y_n\in\RRR \ 
    \text{and a subsequence such that $v_n(\cdot-y_n)\rightharpoonup v$ in $\Hone$}\right\}.
$$
We claim that  $\omega(\{u^0_n\})>0$.  Indeed, 
applying \cite[Lemma I.1]{LionsConcentrationPart1}, if $\omega(\{u^0_n\})=0$, then $u^0_n\to 0$ in $\mathcal{L}^q(\RRR)$ norm, $\forall 2<q<6$, so in particular $\int_{\RRR} (u^0_n)^{\frac{8}{3}}\to 0$.  In addition, by \eqref{condiV} we have 
$\int_{\RRR} V\, |u^0_n|^2 \to 0$, and hence 
$I_V(M)=\lim_{n\to\infty} \EV(u^0_n) \ge 0$, which contradicts Proposition~\ref{properties1}.  Hence ``vanishing'' cannot occur.

We can therefore choose a sequence $x_n^1\in\RRR$ for which $u^0_n(\cdot-x_n^1)\rightharpoonup u^1$, for some $u^1\in\Hone$, with mass 
$m^1:=\|u^1\|_{\mathcal{L}^2(\RRR)}^2 \ge \frac12\omega(\{u^0_n\})>0$.  As $u^0_n\weak 0$, we must have $|x^1_n|\to \infty$.
 In case $m^0=M-m^1$, the sequence converges strongly in $\mathcal{L}^2$, and $u^1$ minimizes $I_V(M-m^1)$, and we obtain \eqref{ca1}, \eqref{splitting1}, \eqref{splitting2}, and \eqref{xblow}, with $N=1$. 
 
If $m^1< M-m^0$, we again define the remainder sequence, $u^1_n(x):=u^0_n(x)-u^1(x+x^1_n)$.  By definition, $u^1_n\weak 0$, $u^1(\cdot-x_n^1)\weak 0$, and $\|u^1_n\|_{\mathcal{L}^2(\RRR)}^2 \to M-m^0-m^1$, and the energy splits,
$$  \Ez(u^0_n) = \Ez(u^1) + \Ez(u^1_n) +o(1)  $$
By the same argument as in \eqref{firststep}, this implies that $\Ez(u^1)=I_0(m^1)$, 
$I_0(M-m^0)=I_0(m^1)+I_0(M-m^0-m^1)$, and $\{u^1_n\}_{n\in\mathbb{N}}$ is a minimizing sequence for $I_0(M-m^0-m^1)$.  Substituting for $I_0(M-m^0)$ in \eqref{firststep} we conclude:
$$  I_V(M)=I_V(m^0)+ I_0(m^1)+I_0(M-m^0-m^1).  $$

We iterate the above process:  for each $k=2,3,\dots$ we obtain translations $\{x^k_n\}$ in $\RRR$, $|x^k_n|\to\infty,$ functions $u^k\in\Hone$ with 
\begin{equation}\label{omega}
 \|u^k\|_{\mathcal{L}^2(\RRR)}^2:= m^k\ge \frac12\omega(\{u^{k-1}_n\}), 
\end{equation}
and remainder sequences 
$$ u^k_n(x):= u^{k-1}_n(x)- u^k(x+x^k_n) = u_n(x)- u^0(x)-\sum_{i=1}^k u^i(x+x^i_n),
$$
satisfying:
\begin{gather}\nnn
\left\| u^k_n\right\|_{\mathcal{L}^2(\RRR)}^2 = M-\left( m^0 +\sum_{i=1}^k m^i\right) +o(1),  \\
u^k_n(\cdot-x^k_n)\weak 0, \quad \text{weakly in $\Hone$,}  \\
\nnn
I_V(M) = I_V( m^0) + \sum_{i=1}^k I_0(m^k),  \quad \text{and hence} \\
\nnn 
\Ez(u^k)= I_V(m^k).
\end{gather}

Next we show that $|x^k_n-x^i_n|\to\infty$ for all $i\neq k$.  Suppose not, and take the {\em smallest} $k>i$ for which $|x^k-x^i|$ remains bounded along some subsequence.  (And so $|x^i_n-x^j_n|\to\infty$ for all $i<j<k$.)  Taking a further subsequence, $(x^k-x^i)\to \xi$ for some $\xi\in\RRR$.  Now note that
$u^i_n(x) = u^k_n(x) + \sum_{j=i+1}^k u^j(x+x^j_n)$, and hence
\begin{equation}\label{xn2}  u^i_n(x-x^i_n)
    = u^k_n(x-x^i_n) + u^k(x-x^i_n+x^i_n) +\sum_{j=i+1}^k u^j(x-x^i_n+x^j_n).
\end{equation}
Since $|x^j_n-x^i_n|\to\infty$ for $i<j<k$, $u^j(\cdot-x^i_n+x^j_n)\weak 0$, while
$u^k(\cdot-x^i_n+x^k_n)\to u^k(\cdot +\xi)$.  And $u^k_n(\cdot-x^i_n)\weak 0$, and hence we pass to the limit in \eqref{xn2} to obtain 
$u^i_n(\cdot-x^i_n)\weak u^k(\cdot+\xi)\neq 0$, a contradiction.  Hence \eqref{xblow} is verified.

We claim that this process must terminate at some finite step $k=N$, for which 
$M=m^0+\sum_{i=1}^N m^i$.  Indeed, if $m^i>0$ for all $i\in\mathbb{N}$, since $M\ge \sum_{i=0}^k m^i$ for all $k$, we have $\lim_{k\to\infty}m^k=0$.  
By \eqref{omega} we conclude that $\lim_{j\to\infty}\omega(\{u^j_n\})=0$, i.e., the remaining mass after $k$ steps,
$(M-\sum_{i=0}^k m^i)$ may be made arbitrarily small.  However, by the concavity of $I_0(M)$ for small (see Appendix or \cite[Lemma 9 (iii)]{NamVBosch},) there exists $M_c>0$ such that minimizing sequences for $I_0$ do not split for $M<M_c$.
This proves statements \eqref{ca1}, \eqref{splitting1}, and \eqref{splitting2}.

For $V\not\equiv 0$, we now show that $m^0>0$, and hence the translations $x_n^0=0$ are trivial in this case.  Indeed, if $m^0=0$, consider the sequence
$\tilde u_n=u_n(x-x^1_n)$.  As $\Ez$ is translation invariant, and $\EV(u^1)<\Ez(u^1)$, a simple calculation shows $\lim_{n\to\infty}\EV(\tilde u_n)< I_V$, which is not possible.
For $V\equiv 0$, the functional $\Ez$ is translation invariant.  Hence, we may begin the process at the Step $k=1$, defining $\omega(\{u_n\})$ and identifying a first set of translates $\{x^0_n\}$ as above.  By translation invariance, $\tilde u_n=u_n(\cdot-x^1_n)$ is also a minimizing sequence for $I_0(M)$, and the weak limit $u^0=w-\lim_{n\to\infty} \tilde u_n$ will be nontrivial.  The rest of the proof continues as in the case of nontrivial $V$.

It remains to show that each $u^i$ solves the Euler-Lagrange equation with the same Lagrange multiplier $\mu$.  By the Ekeland Variational Principle \cite{Ekeland} (see also \cite[Corollary 5.3]{Struwe},)  we may find a minimizing sequence $\{v_n\}_{n\in\mathbb{N}}$, with
$\|v_n\|_{\mathcal{L}^2(\RRR)}^2=M$ and $\|v_n-u_n\|\to 0$, for which the Euler-Lagrange equation is solved up to an small error in $\mathcal{H}^{-1}(\RRR)$.  That is, $\exists\mu_n\in\R$ with
$$   D\EV(v_n) - \mu_n v_n \to 0 \quad\text{in $\mathcal{H}^{-1}$ norm.}
$$
The Lagrange multipliers may be expressed as:
$$  \mu_n M = \langle D\EV(v_n), v_n\rangle + o(1)\|v_n\|_{\Hone}
    = \langle D\EV(v_n), v_n\rangle + o(1),
$$
as minimizing sequences are bounded.  By Lemma~\ref{LemaBoundsun}, $|\mu_n|$ is uniformly bounded, and hence (after extracting a sequence) we may assume $\mu_n\to \mu$ for some $\mu\in\R$.  As $u_n(\cdot-x^i_n)\weak u^i$ weakly in $\Hone$, the same is true for $\tilde v_n:=v_n(\cdot-x^i_n)\weak u^i$, $i=0,\dots,N$.  Hence, for every $\varphi\in C_0^\infty(\RRR)$,
$$   \langle D\EV(u^0)-\mu u^0, \varphi\rangle 
    = \lim_{n\to\infty} \langle D\EV(\tilde v_n)-\mu_n \tilde v_n, \varphi\rangle =0,  $$
and similarly, $D\Ez(u^i)-\mu u^i=0$, $i=1,\dots,N$.
\end{proof}

\section{Appendix:  Concavity for small mass}

We show that $I_0(M)$ is concave for small values of $M$.  This is another property which the TFDW-type functionals share with Gamow's liquid drop model.

\begin{Prop}$I_0(M)$ is strictly concave for $M$ sufficiently small.\end{Prop}
\begin{proof} 
Nam and Van Den Bosch ~\cite{NamVBosch} showed $I_0(M)$ is attained for $M$ small enough by exploiting 
\begin{align} I_0(M)=\inf\{F_u(M);u\in\mathcal{H}^1(\RRR),\vert\vert u\vert\vert_{\mathcal{L}^2(\RRR)}=1\},\quad  M>0,\end{align}
where
\begin{align} F_u(M):=-\frac{M^{\frac{5}{3}}\left(C_u-M^{\frac{2}{3}}D_u\right)_+^2}{4\left(A_u+M^{\frac{2}{3}}B_u\right)}, \end{align}
with
\begin{align} A_u&:=\int_{\RRR}\vert\nabla u( x )\vert^2d x , \ \quad B_u:=c_1\int_{\RRR}\vert u( x )\vert^{\frac{10}{3}}d x ,\\ C_u&:=c_2\int_{\RRR}\vert u( x )\vert^{\frac{8}{3}}d x ,\quad D_u:=\frac{1}{2}\int_{\RRR}\int_{\RRR}\frac{ u^2( x ) u^2( y ) }{\vert x - y \vert}d x d y . \end{align}
Indeed, they proved  $M\ll1\mapsto I_0(M)$ is strictly subadditive by showing that $M\ll1\mapsto F_u(M)/M$ is strictly increasing, uniformly in $u$. The latter was established by making use of the inequalities 
\begin{align}\label{IneqTermsABD}B_u\leq CM^{\frac{2}{3}}A_u, \ \ D_u\leq CM^{\frac{2}{3}}C_u,\quad   u\in \mathcal{H}^1(\RRR).\end{align}
\eqref{IneqTermsABD} follow from H\"{o}lder's inequality, Sobolev's inequality, Hardy-Littlewood's inequality, and the interpolation inequality in Lebesgue spaces. \newline

Then, let us fix any $\alpha\in(0,1)$, and $M_1,M_2\ll1$. By \eqref{IneqTermsABD},
\begin{align}\frac{d^2F}{dM^2}=\frac{d^2}{dM^2}\left[-\frac{M^{\frac{5}{3}}\left(C_u-M^{\frac{2}{3}}D_u\right)^2}{4\left(A_u+M^{\frac{2}{3}}B_u\right)}\right]=-\frac{2G_u(M)}{9M^{\frac{1}{3}}(A_u+M^{\frac{2}{3}}B_u)^3},\quad  M\ll1,\end{align}
where
\begin{align} G_u(M)&:=14M^{\frac{8}{3}}B_u^2D_u^2+M^2(37A_uB_uD_u^2-10B_u^2C_uD_u)+M^{\frac{4}{3}}(27A_u^2D_u^2-30A_uB_uC_uD_u) \\
&\ \ \ +M^{\frac{2}{3}}(-28A_u^2C_uD_u+A_uB_uC_u^2)+5A_u^2C_u^2\\ 
&>A_u^2C_u^2(-10M^2-30M^{\frac{4}{3}}-28M^{\frac{2}{3}}+5)>0,\end{align}
uniformly in $u$. Therefore, $F_u(M)$ is strictly concave for $M\ll1$ uniformly in $u$. In consequence, for some $u*=u*_{\alpha,M_1,M_2}$
\begin{align} I_0(\alpha M_1+(1-\alpha)M_2)&=F_{u^*}(\alpha M_1+(1-\alpha)M_2)\\
&>\alpha F_{u^*}(M_1)+(1-\alpha)F_{u^*}(M_2)\\
&\geq\alpha I_0(M_1)+(1-\alpha)I_0(M_2). \end{align}
Since $\alpha$, $M_1$, and $M_2$ were arbitrary, we conclude that $I_0$ is strictly concave for $M\ll1$.\end{proof}

\end{document}